\documentclass[runningheads]{llncs}

\usepackage{amsmath}
\usepackage{amsfonts}
\usepackage{amssymb}
\usepackage{algorithm}

\usepackage{setspace}
\usepackage[misc,geometry]{ifsym}
\usepackage{cite}

\usepackage{hyperref}
\usepackage{cleveref}
\usepackage{algpseudocode}
\usepackage{subcaption}
\usepackage{tikz-layers}

\newcommand\yes{\textsc{Yes}}
\newcommand\no{\textsc{No}}
\newcommand\np{\textsc{NP}}
\newcommand{\set}[1]{\ensuremath{ \left\lbrace #1 \right\rbrace }}

\newcommand{\wps}{\textsc{Weighted Positive \mbox{2-SAT}}}
\newcommand\contracindep{\textsc{Con\-trac\-tion Blocker($\alpha$)}}
\newcommand\dcontracindep{\textsc{$d$-Con\-trac\-tion Blocker($\alpha$)}}
\newcommand\contracclique{\textsc{1-Con\-trac\-tion Blocker($\omega$)}}
\newcommand\delindep{\textsc{1-Dele\-tion Blocker($\alpha$)}}

\newcommand{\mvc}{\textsc{Vertex Cover}}
\newcommand{\mis}{\textsc{Inde\-pendent Set}}
\newcommand{\dist}{\textrm{dist}}
\DeclareMathOperator*{\merge}{merge}
\DeclareMathOperator*{\val}{val}

\usetikzlibrary{arrows,automata}
\usetikzlibrary{decorations.pathreplacing, calligraphy}
\usetikzlibrary{matrix}
\usetikzlibrary{shapes.geometric}
\tikzstyle{vertex}=[thick,circle,inner sep=0.cm, minimum size=2mm, fill=white, draw=black]
\tikzstyle{hedge}=[thick]
\tikzstyle{gedge}=[thick, draw = ForestGreen, opacity = 1]
\tikzstyle{bcircle}=[thick, circle, minimum size = 15pt, draw]

\newcommand{\gadget}[2]{
	\def\radi{#2}
	\node[] at (0,0){$K_{#1}$};
	
	\draw[thick] (0,0) circle(#2);
	\tangent{0}{0}{\radi}{0}{2.4*\radi}
	\node[vertex, label = left:$v_{#1}$] (v) at (0, 2.4*\radi){};

}
\newcommand{\tangent}[5]{
	\def\r{#3};
	\def\cx{#1};
	\def\cy{#2};
	\pgfmathsetmacro\dx{#4-#1};
	\pgfmathsetmacro\dy{#5-#2};
	
	\pgfmathsetmacro{\lend}{\dx*\dx+\dy*\dy}
	\pgfmathsetmacro{\tmp}{{sqrt(\lend-\r*\r)}}
;
	
	\pgfmathsetmacro{\xa}{\cx + (\dx*\r-\dy*\tmp)*\r/\lend}
	\pgfmathsetmacro{\ya}{\cy + (\dy * \r + (\dx * \tmp))* \r / \lend}
	\pgfmathsetmacro\xb{\cx+ (\dx*\r+\dy*\tmp)*\r/\lend};
	\pgfmathsetmacro\yb{\cy+ (\dy*\r-\dx*\tmp)*\r/\lend};
	\draw[hedge](\xa,\ya) -- (#4,#5);
	\draw[hedge](\xb,\yb) -- (#4,#5);
	
}

\begin{document}

%
%
\author{Felicia Lucke\orcidID{0000-0002-9860-2928} \and
Felix Mann({\scriptsize\Letter})\orcidID{0000-0003-0016-4024}}
\authorrunning{F. Lucke and F. Mann}
%
\institute{Université de Fribourg, Boulevard de Pérolles 90, 1700 Fribourg, Switzerland
\email{\{felicia.lucke,felix.mann\}@unifr.ch}}
%

\title{Reducing Graph Parameters\break by Contractions and Deletions}
\titlerunning{Reducing Graph Parameters by Contractions and Deletions}

\maketitle

\begin{abstract}
We consider the following problem: for a given graph $G$ and two integers $k$ and $d$, can we apply a fixed graph operation at most $k$ times in order to reduce a given graph parameter $\pi$ by at least $d$?
We show that this problem is NP-hard when the parameter is the independence number and the graph operation is vertex deletion or edge contraction, even for fixed $d=1$ and when restricted to chordal graphs. 
We give a polynomial time algorithm for bipartite graphs when the operation is edge contraction, the parameter is the independence number and $d$ is fixed.
Further, we complete the complexity dichotomy on $H$-free graphs when the parameter is the clique number and the operation is edge contraction by showing that this problem is NP-hard in $(C_3+P_1)$-free graphs even for fixed $d=1$. 
When the operation is edge deletion and the parameter is the chromatic number, we determine the computational complexity of the associated problem on cographs and complete multipartite graphs.
Our results answer several open questions stated in [Diner et al., Theoretical Computer Science, 746, p. 49-72 (2012)].

\end{abstract}

\section{Introduction}

Blocker problems are a type of graph modification problems which are characterised by a set $\mathcal{O}$ of graph modification operations (for example vertex deletion or edge contraction), a graph parameter $\pi$ and an integer threshold $d\geq 1$. 
The aim of the problem is to determine, for a given graph $G$, the smallest sequence of operations from $\mathcal{O}$ which transforms $G$ into a graph $G'$ such that $\pi(G')\leq \pi(G)-d$.


As in the case of regular graph modification problems, we often consider a set of operations consisting each of a single graph operation, typically vertex deletion, edge contraction, edge addition or edge deletion. Amongst the parameters which have been studied are the chromatic number $\chi$ (see \cite{Chromatic}), the matching number $\mu$ (see \cite{Matching}), the length of a longest path (see \cite{Paths1,Paths2}), the (total or semitotal) domination  number $\gamma$ ($\gamma_t$ and $\gamma_{t2}$, respectively) (see \cite{TotalDom,domination,SemiTotalDom}), the clique number~$\omega$~(see \cite{clique}) and the independence number $\alpha$ (see \cite{Stable}).

In this paper, the set of allowed graph operations will always consist of only one operation, either \emph{vertex deletion}, \emph{edge contraction} or \emph{edge deletion}.
Given a graph $G$, we denote by $G-U$ the graph from which a subset of vertices $U\subseteq V(G)$ has been deleted. 
Given an edge $uv\in E(G)$, contracting the edge $uv$ means deleting the vertices $u$ and $v$ and replacing them with a single new vertex which is adjacent to every neighbour of $u$ or $v$.
We denote by $G/S$ the graph in which every edge from an edge set $S\subseteq E(G)$ has been contracted.
Further, we denote by $G-S$ the graph $G$ from which a subset of edges $S \subseteq E(G)$ has been deleted.
We consider the following problems, where $d\geq 1$ is a fixed integer.

\begin{center}
\fbox{
\begin{minipage}{4.5in}
\textsc{$d$-Deletion Blocker ($\pi$)}
\begin{description}
\item[Instance:] A graph $G$ and an integer $k$.
\item[Question:] Is there a set $U \subseteq V(G)$, $|U| \leq k$, such that \\$\pi(G-U) \leq \pi(G)-d$? 
\end{description}
\end{minipage}}

\fbox{
\begin{minipage}{4.5in}
\textsc{$d$-Contraction Blocker ($\pi$)}
\begin{description}
\item[Instance:] A graph $G$ and an integer $k$.
\item[Question:] Is there a set $S \subseteq E(G)$, $|S| \leq k$, such that $\pi(G /S) \leq \pi(G)-d$? 
\end{description}
\end{minipage}}

\fbox{
\begin{minipage}{4.5in}
\textsc{$d$-Edge Deletion Blocker ($\pi$)}
\begin{description}
\item[Instance:] A graph $G$ and an integer $k$.
\item[Question:] Is there a set $S \subseteq E(G)$, $|S| \leq k$, such that $\pi(G -S) \leq \pi(G)-d$? 
\end{description}
\end{minipage}}
\end{center}

When $d$ is not fixed but part of the input, the problems are called \textsc{Deletion Blocker}($\pi$), \textsc{Contraction Blocker}($\pi$) and \textsc{Edge Deletion Blocker($\pi$)}, respectively.

When $\pi=\alpha$ or $\pi=\omega$, we know from \cite{CliqueLatest} that \textsc{Deletion Blocker}($\pi$) and \textsc{Contraction Blocker}($\pi$) are NP-hard on general graphs. From \cite{edgedeletion} we know that \textsc{Edge Deletion Blocker($\chi$)} is NP-hard on general graphs. So it is natural to ask if these problems remain NP-hard when the input is restricted to special graph classes.

\begin{table}[ht]
\caption{\label{ComplexTable}The table of complexities for some graph classes. Here, P means solvable in polynomial time, whereas NP-h and NP-c mean NP-hard and NP-complete, respectively. A question mark means that the case is open. Everything in \textbf{bold} are new results from this paper, all other cases are referenced in \cite{CliqueLatest}, where an older version of this table is given.}
\centering
\begin{tabular}{| llllll |}
\hline
Class       & \multicolumn{2}{l}{\textsc{Contraction Blocker}($\pi$)} &  & \multicolumn{2}{l|}{\textsc{Deletion Blocker}($\pi$)} \\ 
            & $\pi=\alpha$               & $\pi=\omega$               &  & $\pi=\alpha$              & $\pi=\omega$             \\ \hline\hline
Tree        & \textsc{P}                          & \textsc{P}                          &  & \textsc{P}                         & \textsc{P}                        \\\hline
Bipartite   & \textsc{NP}-h;                       & \textsc{P}                          &  & \textsc{P}                         & \textsc{P}                        \\
					&  \hspace{4pt}\textbf{d fixed: \textsc{P}}  &  & &   &\\\hline
Cobipartite & $d=1$: \textsc{NP}-c                  & \textsc{NP}-c;            &  & \textsc{P}                         & \textsc{P}                        \\
& & \hspace{4pt} $d$ fixed: \textsc{P} & & &\\\hline
Cograph     & \textsc{P}                          & \textsc{P}                          &  & \textsc{P}                         & \textsc{P}                        \\\hline
Split       & \textsc{NP}-c;           & \textsc{NP}-c;           &  & \textsc{NP}-c;          & \textsc{NP}-c;          \\
				& \hspace{4pt} $d$ fixed: \textsc{P} &  \hspace{4pt}$d$ fixed: \textsc{P} & &  \hspace{4pt}$d$ fixed: \textsc{P} &  \hspace{4pt}$d$ fixed: \textsc{P} \\\hline
Interval    & {?}                          & \textsc{P}                          &  & {?}                           & \textsc{P}                        \\\hline
Chordal     & \textbf{d=1: \textsc{NP}-c}                      & $d=1$: \textsc{NP}-c                  &  & \textbf{d=1: \textsc{NP}-c}                      & $d=1$: \textsc{NP}-c                \\\hline
Perfect     & $d=1$: \textsc{NP}-h                  & $d=1$: \textsc{NP}-h                  &  & \textbf{d=1: \textsc{NP}-c}                      & $d=1$: \textsc{NP}-c                \\ \hline
\end{tabular}

\end{table}

The authors of \cite{CliqueLatest} show that \textsc{Contraction Blocker}($\alpha$) in bipartite and chordal graphs as well as \textsc{Deletion Blocker}($\alpha$) in chordal graphs are \textsc{NP}-hard when the threshold $d$ is part of the input.
However, as an open question, they ask for the complexity of both problems when $d$ is fixed.
In this paper, we show that \textsc{Contraction Blocker}($\alpha$) in bipartite graphs is solvable in polynomial time if $d$ is fixed and that both problems are \np-hard on chordal graphs even if $d=1$. 
An overview of the complexities in some graph classes is given in \Cref{ComplexTable}.

A \emph{monogenic} graph class is characterised by a single forbidden induced subgraph $H$. 
For a given graph parameter $\pi$, it is interesting to establish a \emph{complexity dichotomy for monogenic graphs}, that is, to determine the complexity of \textsc{\mbox{($d$-)Deletion} Blocker($\pi$)} or \textsc{($d$-)Contraction Blocker($\pi$)} in $H$-free graphs, for every graph $H$.
For example, such a dichotomy has been established for \textsc{Deletion Blocker}($\pi$) for all $\pi\in\set{\alpha,\omega,\chi}$ and \textsc{Contraction Blocker}($\pi$) for $\pi\in\set{\alpha,\chi}$ (all \cite{CliqueLatest}), \textsc{Contraction Blocker}($\gamma_{t2}$) (for $d=k=1$, \cite{SemiTotalDom}), \textsc{Contraction Blocker}($\gamma_{t}$) (for $d=k=1$, \cite{TotalDom}) and \textsc{Contraction Blocker}($\gamma$) (for $d=k=1$, \cite{domination}).
In \cite{CliqueLatest}, the computational complexity of \textsc{Contraction Blocker}($\omega$) in $H$-free graphs has been determined for every $H$ except $H=C_3+P_1$. 
We show that this case is \np-hard even when $d=1$ and complete hence the dichotomy.
For the problem \textsc{Edge Deletion Blocker}($\chi$), the authors of \cite{CliqueLatest} observe that the complexity of the problem is known for $H$-free graphs for all $H$ except $H=P_4$ and $H=P_2+P_1$.
We show that \textsc{Edge Deletion Blocker}($\chi$) is \np-complete on $(P_1+P_2)$-free graphs (and thus, for $P_4$-free graphs as well).
For $P_4$-free graphs, we show that the problem is solvable in polynomial time when the difference between $d$ and the chromatic number of the input graph is bounded.
We can also solve $d$-\textsc{Edge Deletion Blocker}($\chi$) on $P_4$-free graphs in polynomial time for any fixed $d$.


\section{Preliminaries}\label{preliminaries}

Throughout this paper, we assume that all graphs are connected unless stated differently. 

We refer the reader to \cite{diestel} for any terminology not defined here.

For any natural number $n$, we denote by $[n]$ the set $\set{1,\ldots,n}$ and by $[0..n]$ the set $\set{0,\ldots,n}$.
For a graph $G$ we denote by $V(G)$ the vertex set of the graph and by $E(G)$ its edge set.
For two graphs $G$ and $H$ we denote by $G+H$ the disjoint union of $G$ and $H$.
For two graphs $G$ and $H$ with disjoint vertex sets, we denote by $G \times H$ the graph with vertex set $V(G) \cup V(H)$ and edge set $E(G) \cup E(H) \cup \set{uv|u \in V(G), v \in V(H)}$.
For two vertices $u,v \in V(G)$, we denote by $\dist_G(u,v)$ the \emph{distance} between $u$ and $v$, which is the number of edges in a shortest path between $u$ and $v$. 
For two sets of vertices $U,W\subseteq V(G)$, the \emph{distance between $U$ and $W$}, denoted by $\dist_G(U,W)$, is given by $\min_{u\in U, w\in W}\dist_G(u,w)$. 
For a set of edges $S\subseteq E(G)$ we denote by $V(S)$ the set of vertices in $V(G)$ which are endpoints of at least one edge of $S$.
Let $v\in V(G)$, then the \emph{(open) neighbourhood of $v$}, denoted by $N_G(v)$, is the set $\set{u\in V(G):\,\dist_G(u,v) = 1}$.
The \emph{closed neighbourhood of $v$}, denoted by $N_G[v]$, is the set $N_G(v)\cup\set{v}$. 
For a set $U \subseteq V(G)$, we define the (open) neighbourhood of $U$ as $N_G(U)=\bigcup_{v\in U}N(v)$ and the closed neighbourhood of $U$ as $N_G[U] = N_G(U)\cup U$. 
If the graph $G$ is clear from the context, we can omit the index.
For a vertex $v\in V(G)$ and a set of vertices $U\subseteq V(G)$, we say that $v$ \emph{is complete to} $U$ if $v$ is adjacent to every vertex of $U$.
Let $G$ be a graph and $S\subseteq E(G)$.
We denote by $G\big\vert_S$ the graph whose vertex set is $V(G)$ and whose edge set is $S$.
For any $U\subseteq V(G)$, we denote by $G[U]$ the subgraph of $G$ induced by $U$. 
For any $U\subseteq V(G)$, we denote by $G-U$ the graph $G[V(G)\setminus U]$.
For any vertex $v\in V(G)$, we denote by $G-v$ the graph $G-\set{v}$.

Let $S \subseteq E(G)$. 
We denote by $G-S$ the graph with vertex set $V(G)$ and edge set $E(G)\setminus S$.
Further, we denote by $G/S$ the graph whose vertices are in one-to-one correspondence to the connected components of $G\big\vert_S$ and two vertices $u,v\in V(G/S)$ are adjacent if and only if their corresponding connected components $A,B$ of $G\big\vert_S$ satisfy $\dist_G(V(A),V(B))=1$.
This is equivalent to the regular notion of contracting the edges in $S$. 
However, this definition allows us to make the notation in the proofs simpler and less confusing.

An \emph{$h$-colouring} of $G$ is a map from $V(G)$ to $[h]$.
For an $h$-colouring $c$ of $G$ and a set $U\subseteq V(G)$, we denote by $c(U)$ the set $\bigcup_{v\in U}c(v)$.

We say that a set $I \subseteq V(G)$ is \emph{independent} if the vertices contained in it are pairwise non-adjacent. 
We denote by $\alpha(G)$ the size of a maximum independent set in $G$. 
The decision problem \mis \ takes as input a graph $G$ and an integer $k$ and outputs \yes\, if and only if there is an independent set of size at least $k$ in $G$. 
We say that a set $U\subseteq V(G)$ is a \emph{clique} if every two vertices in $U$ are adjacent.
We denote by $\omega(G)$ the size of a maximum clique in $G$.
We call a set $U \subseteq V(G)$ a \emph{vertex cover}, if for every edge $uv \in E(G)$ we have that $u \in U$ or $v \in U$. 
The decision problem \mvc \ takes as input a graph $G$ and an integer $k$ and outputs \yes\, if and only if there is a vertex cover of size at most $k$ in $G$.  
We denote by $\tau(G)$ the size of a minimum vertex cover in $G$. 
Furthermore, we call a graph $M$ a \emph{matching} of a graph $G$, if $V(M) \subseteq V(G)$, $E(M) \subseteq E(G)$ and each vertex in $M$ has exactly one neighbour in $M$. 
We say that a matching is a \emph{maximum matching} if it contains the maximum possible number of edges and denote this number by $\mu(G)$.
Observe that we did not use the standard definition of a matching as a set of non-adjacent edges. 
This was done in order to simplify the notation in the proofs.
However, the edge set of a matching in our definition follows the conventional definition.

A graph without cycles is called a \emph{forest} and a connected forest is a \emph{tree}.
It is well-known that a tree has one more vertex than it has edges.
Let $T$ be a tree.
We call a vertex of $T$ a \emph{leaf} if it has exactly one neighbour.
A vertex which is not a leaf is called an \emph{interior} vertex of $T$.
A tree $T$ is called \emph{rooted} if there is a designated vertex called the \emph{root} of $T$.
The children of a vertex $v \in T$ are those neighbours of $v$ whose distance to the root $s$ is larger than $\dist(v,s)$.
A \emph{rooted binary tree} is a rooted tree in which every interior vertex has exactly two children.
A graph is said to be \emph{chordal}, if it has no induced cycle of length at least four.
A graph $G$ is \emph{bipartite}, if we can find a partition of the vertices into two sets $V(G) = U \cup W$ such that $U$ and $W$ are both independent sets.
For a given graph $H$, we say that the graph $G$ is \emph{$H$-free}, if it does not contain $H$ as an induced subgraph.

A graph $G$ is called complete multipartite, if we can partition the vertex set $V(G)$ into disjoint independent sets $I_1, \dots, I_\ell$ and $G$ is isomorphic to $I_1 \times \dots \times I_\ell$.
We call the independent sets $I_1, \dots, I_\ell$ the \emph{parts} of $G$.
It is well known that complete multipartite graphs are exactly the $(P_2 + P_1)$-free graphs.

A graph $G$ is called a \emph{cograph} if one of the following conditions holds:
\begin{itemize}
    \item $G = K_1$,
    \item there are cographs $H, H'$ such that $G = H + H'$, or
    \item there are cographs $H, H'$ such that $G = H \times H'$.
\end{itemize}
A graph is a cograph if and only if it is $P_4$-free (\cite{cographP4}).

Let $T$ be a rooted binary tree with root $s$ whose interior (or non-leaf) vertices are labelled either 0 or 1.
We call the vertices of $T$ \emph{nodes} and the interior vertices \emph{0-} or \emph{1-node}, according to their label.
To every node $p $ of $T$ we associate a cograph $T_p$ as follows:
\begin{itemize}
    \item if $p$ is a leaf, then $T_p = K_1$,
    \item if $p$ is a $0$-node with children $q$ and $r$, then $T_p = T_q + T_r $,
    \item if $p$ is a $1$-node with children $q$ and $r$, then $T_p = T_q \times T_r $.
\end{itemize}
If $T_s$ is isomorphic to a cograph $G$, then we say that $T$ is a \emph{cotree} corresponding to $G$.
For a node $p\in V(T)$, we denote by $T_{\overline{p}}$ the vertex set $V(G)\setminus V(T_p)$.
It was shown in \cite{binarycotree} that every cograph has a corresponding cotree.
Let $p$ be a node of a cograph $T$ with children $q$ and $r$.
It is easy to see that $\chi(T_p)=\max\left\{\chi(T_q),\chi(T_r)\right\}$ if $p$ is a 0-node and $\chi(T_p)=\chi(T_q)+\chi(T_r)$ if $p$ is a 1-node.

For a positive integer $i$, we denote by $P_i$ and $C_i$ the path and the cycle on $i$~vertices, respectively.
We call the graph which is given in Figure \ref{paw} a \emph{paw}.
\begin{figure}
\centering
\begin{tikzpicture}
\node[vertex](c1) at(0,0){};
\node[vertex](c2) at(30:1){};
\node[vertex](c3) at(180:1){};
\node[vertex](c4) at(330:1){};
\draw[hedge](c1)--(c4);
\draw[hedge](c1)--(c3);
\draw[hedge](c1)--(c2);
\draw[hedge](c2)--(c4);
\end{tikzpicture}
\caption{\label{paw}The paw}
\end{figure}

For a given graph parameter $\pi$ we say that a set $S\subseteq E(G)$ is \emph{$\pi$-contraction-critical} if $\pi(G/S)<\pi(G)$.
We say that a set $U\subseteq V(G)$ is \emph{$\pi$-deletion-critical} if $\pi(G-U)<\pi(G)$.

We will use the following two results. The first one is due to K\H{o}nig,  the second one is well-known and easy to see.
\begin{lemma}[see \cite{diestel}]
\label{Koenig}
Let $G$ be a bipartite graph. Then $\mu(G) = \tau(G)$. 
\end{lemma}
\begin{lemma}
\label{mvc+indset}
Let $G$ be a graph and let $I\subseteq V(G)$ be a maximum independent set. Then $V(G) \setminus I$ is a minimum vertex cover and hence $\tau(G) +\alpha(G)= |V(G)|$. 
\end{lemma}
In \cite{poljak} it was shown that \mis \ is NP-complete in $C_3$-free graphs. 
This and \Cref{mvc+indset} imply the following corollary.

\begin{corollary}
\label{mvcNPcomp}
\mvc \ is NP-complete in $C_3$-free graphs.
\end{corollary}

\section{Edge Contractions}\label{Sec:EdgeContraction}
\subsection{Algorithms}
\label{algorithms}

In this section we give a polynomial-time algorithm for $d$-\textsc{Contraction\break Blocker}($\alpha$) in bipartite graphs.

\begin{theorem}
\label{maxnumedgecontraction}
Let $G$ be a connected, bipartite graph with $|V(G)| \geq 2d+2$ and $\alpha(G) \geq d+1$, where $d\geq 1$ is an integer. Then $(G,2d+1)$ is a \yes-instance of \dcontracindep.
\end{theorem}

\begin{proof}
Let $G$ be a bipartite graph with $|V(G)| \geq 2d+2$ and $\alpha(G) \geq d+1$. 
Let $M$ be a maximum matching of $G$. 
Since $G$ is connected, $M$ is non-empty. 
Consider the following algorithm which constructs a tree $T$, which is a subgraph of $G$.

\begin{algorithm}
\caption{}
\label{Algobuildtree}
\begin{algorithmic}[2]
\Require A bipartite graph $G$, a maximum matching $M$ in $G$, an integer $d\geq1$
\Ensure A tree $T$
\State Choose an arbitrary edge $uu'\in E(M)$.
\State Set  $V(T) = \{u,u'\}, E(T) = \{uu'\}$.
\While {$|E(T)| \leq2d-1$}
\State Choose two vertices $w\in N_G(T)\setminus V(T)$, and $w'\in N_G(w)\cap V(T)$. \label{algoline:choosev}
\If {$w\in V(M)$} 	
\State Let $v \in V(M)$ s.t. $ vw \in E(M)$.
\State $V(T) = V(T) \cup \{v,w\}$, $E(T) = E(T) \cup \{w'w,vw\}$ \label{algoline:chooseu}
\Else \ $V(T) = V(T) \cup \{w\}$, $E(T) = E(T) \cup \{w'w\}$
\EndIf
\EndWhile
\State \Return $T$
\end{algorithmic}
\end{algorithm}

We claim that the resulting graph $T$ is a tree. 
Indeed, the initial graph is a single edge and thus a tree. 
Further, observe that every time there are vertices and edges added to $T$ in lines 7 or 8, the resulting graph remains connected and the number of added vertices and added edges is the same. 
It follows that $T$ is connected and has exactly one more vertex than it has edges and is thus a tree.
It is easy to see that $T$ has $2d$ or $2d+1$ edges.

We consider the graph $G' = G-V(T)$.
For every $v \in V(M)\cap V(T)$ the unique vertex $u\in V(M)$ with $uv\in E(M)$ is also contained in $V(T)$ and $uv\in E(T)$.
Thus, there are at most $\Big \lfloor \frac{|V(T)|}{2}\Big \rfloor$ edges in $E(M)$ which have an endvertex in $T$.
Since $M - V(T)$ is a matching in $G'$ we have that $\mu(G') \geq \mu(G)-\Big \lfloor \frac{|V(T)|}{2}\Big \rfloor$.
Applying \Cref{Koenig} and \Cref{mvc+indset}, we get for the independence number of $G'$:
\begin{align*}
\alpha(G') = |V(G')| - \mu(G') &\leq |V(G)|-|V(T)|- \mu(G)+\bigg \lfloor \frac{|V(T)|}{2}\bigg \rfloor \\
&= \alpha(G) - \bigg \lceil \frac{|V(T)|}{2}\bigg \rceil = \alpha(G)-d-1.
\end{align*}
Let $G^*=G/E(T)$. 
Observe that $G\big\vert_{E(T)}$ contains exactly one connected component, say $A$, which has more than one vertex, namely the connected component corresponding to $T$.
Let $v^*\in V(G^*)$ be the vertex which corresponds to $A$.
Since $G^*-v^*$ is isomorphic to $G'$, we obtain that $\alpha(G^*) \leq \alpha(G')+1 \leq \alpha(G)-d$.
\end{proof}

\begin{algorithm}
\caption{}
\label{Algobip}
\begin{algorithmic}[2]

\Require A bipartite graph $G$, an integer $k$, a fixed integer $d$
\Ensure \textsc{Yes} if $(G,k)$ is a \yes-instance of \dcontracindep, \textsc{No} if not
\For {every $S \subseteq E(G)$ of size at most $k$}
	\State Let $\beta = 0$.
	\State Let $G'= G/S$.
	\State Let $ U=\set{v\in V(G'): v\textit{ corresponds to a connected component of }G\big\vert_S \right.$\par\hspace{150pt}  $\left.\textit{ which contains at least 2 vertices} }. $
	\For {every subset $U' \subseteq U$}
		\If {$U'$ is independent }
			\State $\beta = \max(\beta, \alpha(G'-(U\cup N_{G'}(U'))) + |U'|)$ 
		\EndIf
	\EndFor
	\If{$\beta \leq \alpha(G)-d$} 
	\State \Return \textsc{Yes} 
	\EndIf
\EndFor
\State \Return \textsc{No}
\end{algorithmic}
\end{algorithm}

\begin{theorem}
\dcontracindep \ is solvable in polynomial time in bipartite graphs.
\end{theorem}
\begin{proof}
Let $G$ be a bipartite graph and $k$ a positive integer. 
If $|V(G)| \leq 2d+1$ there are at most $2^{d(d+1)}$ subsets of $E(G)$ and at most $2^{2d+1}$ subsets of $V(G)$. 
We can check for every subset $S\subseteq E(G)$ if $\alpha(G/S)\leq\alpha(G)-d$ in constant time by computing the graph $G/S$ and checking for each subset of $V(G/S)$ if it is independent. 
Thus, we can check in constant time if $G$ is a \yes-instance for \dcontracindep.

Since contracting edges in a non-empty graph cannot reduce the number of vertices to zero, it follows that if $\alpha(G) \leq d$ it is not possible to reduce $\alpha(G)$ by $d$ via edge-contractions. 
Hence, we can assume that $|V(G)| \geq 2d+2$ and $\alpha(G) \geq d+1$. 
By \Cref{maxnumedgecontraction}, we know that for $k\geq 2d+1$, it is always possible to contract at most $k$ edges to reduce the independence number of $G$ by at least d, so we can further assume that $k\leq 2d$. 


Consider now \Cref{Algobip} which takes as input $G,k$ and $d$ and outputs $\yes{}$ or $\no{}$.
\Cref{Algobip} considers every subset $S\subseteq E(G)$ of edges of cardinality at most $k$ and computes $\alpha(G/S)$. 
If there is some $S$ such that $\alpha(G/S)\leq\alpha(G)-d$ then we return $\yes$, and $\no$ otherwise. 
In order to compute $\alpha(G/S)$ for such a subset $S$ of edges, we first set $G'=G/S$ and consider the set of vertices $U\subseteq V(G')$ which have been formed by contracting some edges in $S$ (see line 4 of the algorithm). 
Observe that $G[V(G')\setminus U]$ is isomorphic to $G-V(S)$ and induces thus a bipartite graph. 
Every independent set of $G'$ can be partitioned into a set $U'\subseteq U$ and a set $W\subseteq V(G')\setminus (U\cup N_{G'}(U'))$. 
Thus, we can find the independence number of $G'$ by considering every independent subset $U'$ of $U$ and computing $\alpha(G'- (U\cup N_{G'}(U')))+\vert U'\vert$. 
The largest of these values is then $\alpha(G')$. 
The independence number of the bipartite graph $G'-(U\cup N_{G'}(U'))$ can be computed in polynomial time, see \Cref{mvc+indset} and \cite{ahuja}.

The number of subsets of $E(G)$ of cardinality at most $k$ is in $O(\vert E(G)\vert^k)=O(\vert V(G)\vert^{4d})$.
For any such subset $S$, the number of subsets $U'\subseteq U$ is at most $2^k\leq 2^{2d}$.
Thus, the running time of \Cref{Algobip} is polynomial.
\end{proof}
\subsection{Hardness proofs}
\label{hardness}

In this section, we answer several questions asked in \cite{CliqueLatest}. 
Indeed, \break \Cref{C3P1Hard} settles the missing case of \cite[Theorem 24]{CliqueLatest} and completes the complexity dichotomy for $H$-free graphs, which is given after the theorem.
We further settle the computational complexity of $1$-\contracindep{} in chordal graphs, which is an open case in \Cref{ComplexTable}.

\begin{theorem}\label{C3P1Hard}
The decision problem \contracclique \ is NP-hard in $(C_3 + P_1)$-free graphs.
\end{theorem}
\begin{proof}
We use a reduction from \mvc \ in $C_3$-free graphs which is NP-complete due to \Cref{mvcNPcomp}. 
Let $(G,k)$ be an instance of \mvc \ where $G$ is a $C_3$-free graph. 
Since \mvc\ is trivial to solve on a graph without edges, we can assume that $E(G)$ is non-empty. 
We construct an instance $(G',k)$ of \contracclique \  such that $(G,k)$ is a \yes-instance of \mvc \ if and only if $(G',k)$ is a \yes-instance of \contracclique\ and $G'$ is $(C_3+P_1)$-free. 
Let $G'$ be a graph with $V(G') = V(G) \cup \set{w}$, $w\notin V(G)$, and $E(G') = E(G) \cup \{wv, v\in V(G)\}$. 
In other words, we add a universal vertex $w$ to $G$ in order to obtain $G'$.

Since $G$ is $C_3$-free, every copy of $C_3$ in $G'$ has to contain $w$. 
Furthermore, since $w$ is adjacent to every other vertex in $V(G')$, it follows that every vertex of $G'$ has distance at most one to every copy of $C_3$. 
Thus, $G'$ is $(C_3 + P_1)$-free. 
Also, note that $\omega(G') = 3$ and that every maximum clique in $G'$ is a copy of $C_3$ which contains $w$ and exactly two vertices of $V(G)$.

Let us assume that $(G,k)$ is a \yes-instance of \mvc. 
Let \break $\set{v_1,\dots, v_k}\subseteq V(G)$ be a vertex cover of $G$. 
Set $S = \set{v_iw\colon\,i \in \set{1,\dots, k}} $ and let $G^* = G'/S$. 
We claim that $S$ is $\omega$-contraction-critical. 
Notice that the contraction of an edge $vw \in S$ is equivalent to deleting the vertex $v$, since the new vertex remains adjacent to all other vertices. 
Thus, $G^*$ is isomorphic to $G-(V(S)\setminus\set{w})$. 
Since $\set{v_1,\dots, v_k}$ is a minimum vertex cover of $G$, there are no edges in $G^*-w$, meaning that $G^*$ is $C_3$-free and thus $\omega(G^*) \leq 2$. 
Hence $(G',k)$ is a \yes-instance of \contracclique.

For the other direction, assume that $(G',k)$ is a \yes-instance of \contracclique.
Let $S\subseteq E(G')$ be a minimum $\omega$-contraction-critical set of edges with $\vert S\vert\leq k$ and let $G^* = G'/S$.

We construct a set $U$ of vertices of $G$ as follows: For the connected component $T$ of $G'\big\vert_S$ that contains $w$, add every vertex of $V(T)$ except $w$ to $U$. 
For every other connected component $T$ of $G'\big\vert_S$ we add to $U$ all vertices of $V(T)$ except one, which can be chosen arbitrarily.
We claim that $U$ is a vertex cover of $G$ of size at most $k$.

To see that $\vert U\vert\leq k$, observe that for every connected component $T$ of $G'\big\vert_S$ we have added $\vert V(T)\vert -1$ vertices to $U$. 
Since $T$ is a tree (see \Cref{forest}), we have that $\vert V(T)\vert -1=\vert E(T)\vert$. 
Thus, we have added as many vertices to $U$ as there are edges in $S$ and hence $\vert U\vert =\vert S\vert\leq k$.

In order to show that $U$ is a vertex cover, suppose for a contradiction that there is an edge $uv\in E(G)$ for which neither $u$ nor $v$ is contained in $U$. 
Consider the connected components $A_u$, $A_v$ and $A_w$ of $G'\big\vert_S$ which contain $u$, $v$ and $w$, respectively. 
It follows from the construction of $U$ that in every connected component $T$ of $G'\big\vert_S$ there is at most one vertex of $T$ which is not contained in $U$. 
Hence, $A_u\neq A_v$. 
We have that $w\not\in U$ by construction, so the same argument can be used to show that $A_u\neq A_w$ and $A_v\neq A_w$. 
Thus, $A_u$, $A_v$, $A_w$ correspond to three different vertices in $G^*$ and since the components are pairwise at distance one, their corresponding vertices induce a $C_3$ in $G^*$, a contradiction to $S$ being $\omega$-contraction-critical.
Thus, $U$ is a vertex cover in $G$ and $(G,k)$ a \yes-instance of \mvc.
\end{proof}

\begin{theorem}
Let $H$ be a graph. 
If $H$ is an induced subgraph of $P_4$ or of the paw, then \textsc{Contraction Blocker}$(\omega)$ is polynomial-time solvable for $H$-free graphs, otherwise it is NP-hard or co-NP-hard for $H$-free graphs.
\end{theorem}

In order to simplify the notation of the proof of the following theorem, we restate \textsc{Vertex Cover} as a satisfiability problem.

\vspace{7pt}
\begin{center}
\fbox{
\begin{minipage}{4.5in}
\wps
\begin{description}
\item[Instance:] A variable set $X$,  a clause set $C$ in which all clauses contain exactly two literals and every literal is positive, as well as an integer $k$.
\item[Question:] Is there a truth assignment of the variables (that is, a mapping $f\colon\, X\rightarrow \set{\textrm{true},\,\textrm{false}}$) such that at least one literal in each clause is true and there  are at most $k$ variables which are true.
\end{description}
\end{minipage}}
\end{center}

\vspace{7pt}

If $\Phi=(G,k)$ is an instance of \textsc{Vertex Cover} then taking $X=V(G)$ as the variable set and $C=\set{(u \lor w)\colon\, uw\in E(G)}$ as the set of clauses yields an instance $(X,C,k)$ of \wps\ which is clearly equivalent to $\Phi$. 
Since \textsc{Vertex Cover} is known to be NP-hard (see \Cref{mvcNPcomp}), it follows that \wps\ is NP-hard, too.

Let $G$ be a graph and $S, S' \subseteq E(G)$ such that for every connected component $A$ of $G\big\vert_S$ there is a connected component $A'$ of $G\big\vert_{S'}$ with $V(A)=V(A')$. 
Then, $G/S = G/S'$ and thus we get the following corollary.


\begin{corollary}\label{forest}
Let $G$ be a graph and $S\subseteq E(G)$ a minimal $\alpha$-contraction-critical set of edges. 
Then, $G\big\vert_S$ is a forest.
\end{corollary}

\begin{theorem}
\label{contchordal}
$1$-\contracindep{} is \np-complete in chordal graphs.
\end{theorem}
\begin{proof}
It was shown in \cite{gavril} that \mis \ can be solved in polynomial time for chordal graphs. 
Since the family of chordal graphs is closed under edge contractions, for a given chordal graph $G$ and a set $S\subseteq E(G)$, it is possible to check in polynomial time whether $S$ is $\alpha$-contraction-critical. 
It follows that 1-\contracindep{} is in \np\ for chordal graphs. 
In order to show \mbox{\np-hardness}, we reduce from \wps, which was shown to be \np-hard above. 
Let $\Phi=(X,C,k)$ be an instance of \wps. 
We construct a chordal graph $G$ such that $(G,k)$ is a \yes-instance for 1-\contracindep{} if and only if $\Phi$ is a \yes-instance for \wps, as follows:

For every variable $x\in X$, we introduce a set of vertices $G_x$ with $G_x=\set{v_x}\cup K_x$, where $K_x$ is a set of $2k+1$ vertices which induce a clique. 
We make $v_x$ complete to $K_x$. 
For every clause $c\in C$, we introduce a vertex $v_c$. We define $K_C = \bigcup_{c \in C} \set{v_c}$. We add edges so that $G[K_C]$ is a clique. For every clause $c\in C$, $c= (x\lor y)$, we make $v_c$ complete to $K_{x}$ and $K_{y}$ (see Figure 1 for an example).

Observe first that the graph $G$ is indeed chordal: if a cycle of length at least four contains at least three vertices of $K_C$, it follows immediately that the cycle cannot be induced, since $K_C$ induces a clique.
Otherwise, such a cycle contains at most two vertices of $K_C$.
If there are two vertices $w$ and $w'$ of the cycle which are contained in $G_x$ and $G_y$, respectively, with $x,y \in X, x\neq y$, then the cycle has to contain a chord in $G[K_C]$ and is thus not induced. 
If all vertices of the cycle are in $K_C\cup G_x$ for some fixed $x\in X$, then there are at least two vertices $w$ and $w'$ contained in $K_x$.
Hence, the cycle cannot be induced since $w$ and $w'$ are adjacent and have the same neighbourhood.
It follows that $G$ cannot have any induced cycle of length at least 4 and is thus chordal.

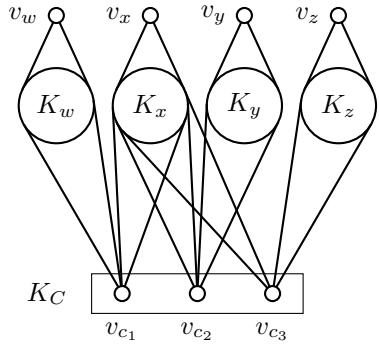
\begin{figure}
\centering
\begin{tikzpicture}
\def\radius{0.5}
\def\k{1}
\def\h{2.5}

\begin{scope}[shift = {(-1.875*\k,\h)}]
	\gadget{w}{\radius}
\end{scope}

\begin{scope}[shift = {(-0.625*\k,\h)}]
	\gadget{x}{\radius}
\end{scope}

\begin{scope}[shift = {(0.625*\k,\h)}]
	\gadget{y}{\radius}
\end{scope}

\begin{scope}[shift = {(1.875*\k,\h)}]
	\gadget{z}{\radius}
\end{scope}


\tangent{-1.875*\k}{\h}{\radius}{-1*\k}{0}

\tangent{-0.625*\k}{\h}{\radius}{0}{0}
\tangent{-0.625*\k}{\h}{\radius}{-1*\k}{0}
\tangent{-0.625*\k}{\h}{\radius}{1*\k}{0}

\tangent{0.625*\k}{\h}{\radius}{0}{0}

\tangent{1.875*\k}{\h}{\radius}{1*\k}{0}

\node[vertex, ] (c3) at (-1*\k,0){};
\node[vertex,] (c4) at (0*\k,0){};
\node[vertex,] (c5) at (1*\k,0){};
  
\node[] (c3t) at (-1*\k,-0.5){$v_{c_1}$};
\node[] (c4t) at (0*\k,-0.5){$v_{c_2}$};
\node[] (c5t) at (1*\k,-0.5){$v_{c_3}$};
  
  
\node[] (Kct) at (-2*\k,0){$K_C$};
\node[rectangle,draw, minimum width = 80*\k , minimum height = 15*\k](r) at (0*\k,0.0){};
  
\end{tikzpicture}

\caption{This is the graph corresponding to the instance of \wps \ given by the variables $w,x,y,z$ and the clauses $c_1 = w\lor x, c_2 = x\lor y$ and $c_3 = x\lor z$. The rectangular box corresponds to $G[K_C]$, the vertices contained in it induce a clique. Every set $K_i$ induces a clique and the lines between a vertex and a set $K_i$ mean that this vertex is complete to $K_i$.}
\end{figure}

Since $G_x$ induces a clique for every $x\in X$, it can contain at most one vertex in any independent set; the same applies to $K_C$. 
Thus, $\alpha(G)\leq \vert X\vert +1$. 
Let $c\in C$. 
Since the set $\{v_x \colon x\in X\}\cup \set{v_c}$ is an independent set of size $\vert X\vert +1$, it follows that $\alpha(G)=\vert X\vert +1$.

Let us assume that $\Phi$ is a \yes-instance of \wps.
Let $X_{+}$ be the set of positive variables of a satisfying assignment of $\Phi$.
For each $x \in X_{+}$, let $e_x$ be an edge incident to $v_{x}$ and let $S=\set{e_x\vert x\in X_{+}}$.
Let $G'=G/S$. We claim that $\alpha(G')<\alpha(G)$.
To see this, observe first that for any $x\in X_{+}$, contracting $e_x$ is equivalent to deleting the vertex $v_{x}$, since $N_G(v_{x})=K_{x}$ induces a clique.
Therefore, we have that $G'\simeq G-\set{v_x:\, x\in X_{+}}$.
Suppose for a contradiction that there is an independent set $I$ of $G'$ of size $|X|+1$.
Since $\vert I\cap K_x\vert\leq 1$ (for $x\in X_{+})$ and $\vert I\cap G_x\vert\leq 1$ (for all $x\in X\setminus X_{+}$), it follows that there exists $c\in C$ such that $v_c\in K_C\cap I$.
Furthermore, the inequalities above all have to be equalities.
By the choice of $X_{+}$, it follows that there is $x\in X_{+}$ such that $x$ is a literal in~$c$.
Since $\vert I\cap K_x\vert=1$, there is a vertex $w\in I\cap K_x$ which is adjacent to $v_c$, contradicting the fact that $I$ is independent.
It follows that $S$ is $\alpha$-contraction-critical.

For the other direction, assume that $\Phi' = (G,k)$ is a \yes-instance of 1-\contracindep{}. Let $S$ be a minimum $\alpha$-contraction-critical set of edges such that $\vert S\vert\leq k$. 
By \Cref{forest}, the graph $G\big\vert_S$ is a forest.

For any $x\in X$, there is a vertex $u_x\in K_x\setminus V(S)$. 
This follows from the\break fact that $k$ edges can be incident to at most $2k$ vertices and\break $\vert K_x\vert =2k+1$.  
Let $H$ be the graph with vertex set $V(H) = K_C$ and edge set $E(H) = \set{uv\in S:u,v \in K_C}$.

Suppose for a contradiction that there is a connected component $T$ of $H$ such that for every $x\in X$ with $\dist_G(G_x,V(T))=1$ we have $G_x\cap V(S)=\varnothing$.
In other words, for every $c=(x\lor y)\in C$ with $v_c\in V(T)$ we have $G_x\cap V(S)=G_y\cap V(S)=\varnothing$.
So we have that $N_G[V(T)]\cap V(S)\subseteq V(T)$, and thus $T$ is also a connected component in $G\big\vert_S$.
For every $x\in X$ the set $\set{u_x}$ is a connected component in $G\big\vert_S$, that is, $u_x$ is not incident to any edge in $S$.
Further, for every $x\in X$ where $\dist_G(G_x,V(T))=1$, we have that $G_x\cap V(S)=\varnothing$ and thus $\set{v_x}$ is a connected component in $G\big\vert_S$.
Let $X_1= \set{x\in X\colon\, \dist_G(u_x,V(T))=1}$ and $X_2=X\setminus X_1$.
Observe that the set $I=T \cup\set{\set{v_x}\colon\, x\in X_1}\cup\set{\set{u_x}\colon\, x\in X_2}$ is a set of connected components of $G\big\vert_S$ which correspond to vertices in $G/S$ who are pairwise at distance at least two.
In other words, $I$ corresponds to an independent set in $G/S$ of cardinality $\vert X\vert+1$, a contradiction to the assumption that $S$ is $\alpha$-contraction-critical.
It follows that there is no connected component $T$ of $H$ such that for every $x\in X$ with $\dist_G(G_x,V(T))=1$ we have $G_x\cap V(S)=\varnothing$.

We can obtain a truth assignment of the variables satisfying $\Phi$ as follows: Set every $x$ to true for which $G_x\cap V(S)$ is non-empty.
For every clause $c=(x\lor y)\in C$ for which both $G_x\cap V(S)$ and $G_y\cap V(S)$ are empty, set one of its variables to true. 
This assignment is clearly satisfying, it remains to show that we set at most $\vert S\vert\leq k$ variables to true. 
Consider a connected component $T$ of $H$. 
Recall that $T$ is a tree, and so its number of vertices is one more than its number of edges. 
We have shown that there is a vertex $v_c\in V(T)$, $c=(x\lor y)$, for which $G_x\cap V(S)\neq\varnothing$.
Thus, there are at most $\vert E(T)\vert$ vertices $v_c\in T$, $c=(x\lor y)$, for which both $G_x\cap V(S)$ and $G_y\cap V(S)$ are empty.
This implies that for every connected component $T$ of $H$ we set at most $\vert E(T)\vert$ variables to true.
Further, the number of variables $x\in X$ which we set to true because $G_x\cap V(S)\neq\varnothing$ is at most the number of edges of $S$ which are not contained in $G[K_C]$.
This shows that, in total, we set at most $\vert S\vert$ variables to true, which concludes the proof.
\end{proof}

\section{Vertex Deletions}\label{Sec:VxDel}

In this section, we settle another open case of \Cref{ComplexTable}.
Interestingly, \delindep{} and 1-\contracindep{} are equivalent on the instance $\Phi'$ constructed in the proof of \Cref{contchordal} and thus the same construction can be used to show \np-hardness of \delindep{} in chordal graphs.

\begin{theorem}
\label{delchordal}
\delindep{} is \np-complete in chordal graphs.
\end{theorem}
\begin{proof}
It has been shown in \cite{gavril} that it is possible to determine the independence number of chordal graphs in polynomial time. 
Since chordal graphs are closed under vertex deletion, it is possible to check in polynomial time whether the deletion of a given set of vertices reduces the independence number. 
Hence \mbox{\delindep{}} is in \np \ for chordal graphs.

In order to show \np-hardness, we reduce from \wps. Let $\Phi$ be an instance of \wps,  $\Phi=(X, C,k)$. 
Let\break $\Phi'=(G,k)$ be the instance of 1-\contracindep{} which is described in \Cref{contchordal} and which has been shown to be equivalent to $\Phi$. 
Further, let $K_x, G_x$ and $v_x$ for each $x\in X$, $K_C$, and $v_c$ for each $c\in C$ be as in the proof of \Cref{contchordal}. 
Recall that we have shown that $\alpha(G)=\vert X\vert +1$ and that $G$ is chordal.

We show that $\Phi'$ is a \yes-instance of \delindep{} if and only if $\Phi$ is a \yes-instance of \wps.

Assume first that $\Phi$ is a \yes-instance of \wps \ and that $X_{+}$ is the set of positive variables in a satisfying assignment of $\Phi$. We have shown in the proof of \Cref{contchordal} that $\alpha(G- \set{v_x\colon x\in X_{+}})<\alpha(G)$, hence $(G,k)$ is a \yes-instance of \delindep{}.

Conversely, assume that $\Phi'$ is a \yes-instance of \delindep{} and let $W$ be an $\alpha$-deletion-critical set of vertices of cardinality at most $k$. 
For every $x\in X$ there is $u_x\in K_x\setminus W$, since $\vert W\vert <\vert K_x\vert$.
Define a set\break $Z=\set{x\in X\colon v_x\in W}$ and initialize a set $Z'=\varnothing$. 
For every clause $c\in C$ with $v_c \in W$ we choose one of the variables contained in $c$ and add it to $Z'$. 
We claim that setting the variables of $Z\cup Z'$ to true yields a satisfying assignment of $\Phi$. 
Observe first that $\vert Z\cup Z'\vert\leq\vert W\vert\leq k$ by construction. 
Suppose for a contradiction that there is a clause $c\in C$, $c=(x\lor y)$, such that neither $x$ nor $y$ is contained in $Z\cup Z'$. 
It follows that $v_{x},v_{y},v_c\notin W$. 
But then $\{v_c,v_{x},v_{y}\} \cup \set{u_z\colon\,z\in X\setminus \set{x,y}}$ is an independent set of size $|X|+1$ in $G-W$, a contradiction to the $\alpha$-deletion-criticalness of $W$. 
Hence the assignment is satisfying and the theorem follows.
\end{proof}

Since perfect graphs are a superclass of chordal graphs, we obtain the following corollary.

\begin{corollary}
\label{delperfect}
\delindep{} is \np-complete in perfect graphs.
\end{corollary}

Observe that \Cref{delperfect} could also be shown as follows.
Complements of perfect graphs are again perfect graphs.
Further, \delindep{} is a \yes-instance for a graph $G$ if and only if $\textsc{1-Dele\-tion Blocker($\omega$)}$ is a \yes-instance for $\overline{G}$.
Since it was shown in \cite{CliqueLatest} that $\textsc{1-Dele\-tion Blocker($\omega$)}$ is \np-hard in perfect graphs the corollary follows.

\section{Edge Deletions}\label{Sec:EdgeDel}

Given a colouring $c$ of the vertices of a graph $G$, we say that an edge $uv$, with $u,v\in V(G)$, is a \emph{monochromatic edge} of $c$ if $c(u)=c(v)$.
Using this terminology, a \emph{proper colouring} is a colouring without monochromatic edges.
The following problem is a generalization of \textsc{$h$-Chromatic Number} (for a definition, see for example \cite{GareyJohnson}), in the sense that we ask if there is a colouring with few monochromatic edges.

\begin{center}
\fbox{
\begin{minipage}{4.5in}
\textsc{$h$-Monochromatic Edges}
\begin{description}
\item[Instance:] A graph $G$ and an integer $m$.
\item[Question:] Is there an $h$-colouring of $G$ with at most $m$ monochromatic edges? 
\end{description}
\end{minipage}}
\end{center}

\vspace{7pt}

As above, we sometimes consider $h$ to be part of the input.
The problem is then called \textsc{Monochromatic Edges}.

To keep the notation more simple, we will focus on \textsc{Monochromatic edges} instead of \textsc{Edge Deletion Blocker}($\chi$) in this chapter.
This is justified by the following proposition.

\begin{lemma}
    The tuple $(G,m)$ is a \yes-instance for \textsc{$h$-Monochromatic Edges} if and only if $(G,m)$ is a \yes-instance for $(\chi(G)-h)\textsc{-Edge Deletion Blocker}(\chi)$.
\end{lemma}
\begin{proof}
    Let $G$ be a graph and $m$ an integer.
    If $(G,m)$ is a \yes-instance for \textsc{$h$\nobreakdash-Monochromatic Edges}, then there is an $h$-colouring $c$ of $G$ with at most $m$ monochromatic edges.
    Let $S\subseteq E(G)$ be the set of monochromatic edges of $c$.
    Then, $c$~is a proper $h$-colouring of $G-S$.
    It follows that deleting $\vert S\vert\leq m$ edges from $G$ yields a graph whose chromatic number is at most $h$.
    In other words, $(G,m)$ is a \yes-instance for $(\chi(G)-h)$-\textsc{Edge Deletion Blocker}($\chi$).
    For the other direction, assume that $(G,m)$ is a \yes-instance for $(\chi(G)-h)$-\textsc{Edge Deletion Blocker}($\chi$).
    Thus, there is a set of edges $S\subseteq E(G)$ such that $\vert S \vert\leq m$ and $\chi(G-S)\leq \chi(G)-(\chi(G)-h)=h$.
    Let $c$ be a proper $h$-colouring of $G-S$.
    When we colour the vertices of $G$ according to~$c$, then the only monochromatic edges can be the edges in $S$.
    Thus, $c$ is an $h$-colouring of $G$ with at most $m$ monochromatic edges, which completes the proof.
\end{proof}

The following lemma is a simple observation about reducing the number of monochromatic edges by recolouring the vertices.

\begin{lemma}\label{Indepsets}
    Let $G$ be a graph, and $I\subseteq V(G)$ an independent set such that for any $v\in I$ we have $N(v)=N(I)$.
    If $c$ is an $h$-colouring of $G$, then there is a colour $j\in c(I)$ such that recolouring every vertex of $I$ with $j$ yields an $h$-colouring of $G$ which has at most as many monochromatic edges as $c$.
\end{lemma}
\begin{proof}
For each $i\in[h]$, we denote by $n_i$ the number of vertices in $I$ which receive colour $i$ by $c$.
Similarly, for every $i\in[h]$, we denote by $n'_i$ the number of vertices in $N(I)$ which receive colour $i$ by $c$.
Since $I$ is independent, no colouring can have any monochromatic edges between two vertices of $I$.
The number of monochromatic edges between $I$ and $N(I)$ is $\sum_{i\in [h]}n_in'_i$.
Let $j\in c(I)$ be such that $n'_{j}$ is minimum amongst all colours in $c(I)$.
After recolouring all vertices in $I$ with colour $j$, the number of monochromatic edges between $I$ and $N(I)$ is $\vert I\vert n'_{j}=\sum_{i\in [h]}n_in'_{j}\leq \sum_{i\in [h]}n_in'_i$.
This concludes the proof.
\end{proof}

\subsection{Algorithms}

\begin{theorem}
    For a fixed integer $h$, the decision problem $h$-\textsc{Monochromatic Edges} is solvable in polynomial time on cographs.
\end{theorem}
\begin{proof}

Let $G$ be a cograph with associated cotree $T$.
        For every $p\in V(T)$, we define a function $f^p$ which takes as input an $h$-tuple of non-negative integers $a^p = (a^p_1,\ldots,a^p_h)$ whose entries sum up to $\vert V(T_p)\vert$.
       
        
        If $p$ is a leaf, then $f^p(a^p)=0$ for any valid input.
        If $p$ is not a leaf, let $q$ and $r$ be the children of $p$.
        If $p$ is a 0-node, then 
        $$f^p(a^p)=\min_{\substack{a^q,a^r\in [0..n]^h\\a_1^q+\ldots+a_h^q=\vert T_q\vert\\a_1^r+\ldots+a_h^r=\vert T_r\vert\\a_i^q+a_i^r=a^p_i,\,i\in[h] }}\left(f^q(a^q)+f^r(a^r)\right).$$
        If $p$ is a 1-node, then 
        $$f^p(a^p)=\min_{\substack{a^q,a^r\in [0..n]^h\\a_1^q+\ldots+a_h^q=\vert T_q\vert\\a_1^r+\ldots+a_h^r=\vert T_r\vert\\a_i^q+a_i^r=a^p_i,\,i\in[h] }}\left(f^q(a^q)+f^r(a^r)+\sum_{i=1}^h a_i^qa_i^r\right).$$
        This defines the values of $f^p$ for every $p\in V(T)$.
        Observe that for every node $p\in V(T)$, there are at most $O(n^h)$ possible inputs $a^p$.
        To compute $f^p(h)$, we consider every pair of $h$-tuples $a^q,a^r\in[0..n]^h$, of which there are $O(n^{2h})$.
        Checking whether they sum to the correct values and computing the term given in the formula above takes constant time.
        So, we can compute the function $f^p$ in polynomial time.
        \\\\
        \textbf{Claim.}
        For every $p\in V(T)$ and every $h$-tuple $a^p = (a^p_1,\ldots,a^p_h)$ of non-negative integers with
        \begin{align}
        a^p_1+\ldots+a^p_h&=\vert T_p\vert,\label{ClaimCondition2}
        \end{align}
        the value of $f^p(a^p)$ is the minimum number of monochromatic edges of all $h$-colourings of $T_p$, in which colour $i$ appears $a^p_i$ times for every $i\in [h]$.
        \\
        \textit{Proof of the claim.}\\
        Observe first that the claim holds when $p$ is a leaf, since then there are no edges in $T_p$.
        If $p$ is not a leaf, let $q$ and $r$ be the children of $p$ and assume that the claim holds for $q$ and $r$.
        
        For an $h$-tuple $a^p = (a^p_1,\ldots,a^p_h)$ which satisfies (\ref{ClaimCondition2}), let $c_p$ be an $h$-colouring of $T_p$ which, for every $i\in[h]$, assigns colour $i$ to exactly $a^p_i$ vertices, and which minimizes the number of monochromatic edges amongst all such colourings.
        Let $m_p$ be the number of monochromatic edges in $c_p$.
        Let $c_q$ and $c_r$ be the $h$-colourings of $T_q$ and $T_r$, respectively, which we obtain by restricting $c_p$ to $T_q$ and $T_r$, respectively.
        For every $i\in [h]$, let $a_i^q$ be the number of vertices in $T_q$ which receive the colour $i$ from $c_q$.
        Define $a_i^r$ analogously for all $i\in[h]$.
        Clearly, $a^p_i=a_i^q+a_i^r$ for every $i\in[h]$ as well as $\sum_{i=1}^ha_i^q=\vert T_q\vert$ and $\sum_{i=1}^ha_i^r=\vert T_r\vert$.
        It follows from the definition of $f^p$ that $f^p(a^p) \leq f^q(a^q)+f^r(a^r)$ if $p$ is a 0-node and $f^p(a^p) \leq f^q(a^q)+f^r(a^r)+\sum_{i=1}^ha_i^qa_i^r$ if $p$ is a 1-node.
        Let $m_q$ be the number of monochromatic edges of $c_q$, and define $m_r$ analogously.
        Since the claim holds for $q$ and $r$, it follows that $f^q(a^q)\leq m_q$ and $f^r(a^r)\leq m_r$.
        If $p$ is a 0-node, then the number of monochromatic edges of $c_p$ is $m_p = m_q+m_r$ and thus, $f^p(a^p) \leq m_q+m_r = m_p$.
        If $p$ is a 1-node, then the number of monochromatic edges of $c_p$ is $m_p = m_q+m_r+\sum_{i=1}^h a_i^qa_i^r$ and thus, $f^p(a^p) \leq m_q+m_r+\sum_{i=1}^h a_i^qa_i^r = m_p$.
        
        Suppose for a contradiction that $f^p(a^p) < m_p$.
        Then there are $h$-tuples of non-negative integers $b^q= (b_1^q,\ldots,b_h^q)$ and $b^r = (b_1^r,\ldots,b_h^r)$ such that
        \begin{itemize}
            \item $b_1^q+\ldots+b_h^q=\vert T_q\vert$,
            \item $b_1^r+\ldots+b_h^r=\vert T_r\vert$,
            \item $b_i^q+b_i^r=a^p_i$ for all $i\in[h]$,
            \item $f^q(b^q)+f^r(b^r)=f^p(a^p)<m_q+m_r$ if $p$ is a 0-node and
            \item $f^q(b^q)+f^r(b^r)+\sum_{i=1}^h b_i^qb_i^r=f^p(a^p)<m_q+m_r+\sum_{i=1}^h a_i^qa_i^r$ if $p$ is a 1-node.
        \end{itemize}
        Since we assume that the claim holds for $q$ and $r$, it follows that there is an $h$-colouring $c'_q$ of $T_q$ such that the number of vertices in $T_q$ which receive the colour $i$ by $c'_q$ is $b_i^q$ for every $i\in[h]$ and whose number of monochromatic edges is $f^q(b^q)$ and analogously for $T_r$.
        Let us colour $T_p$ by colouring the vertices in $T_q$ according to $c'_q$ and the vertices in $T_r$ according to $c'_r$.
        This yields a well-defined $h$-colouring $c_p'$ of $T_p$ since $V(T_p)$ is a disjoint union of $V(T_q)$ and $V(T_r)$.
        The number of monochromatic edges of $c_p'$ is $f^q(b^q)+f^r(b^r)=f^p(a^p)$ if $p$ is a 0-node and $f^q(b^q)+f^r(b^r)+\sum_{i=1}^hb_i^qb_i^r=f^p(a^p)$ if $p$ is a 1-node.
        Since $f^p(a^p)< m_p$ it follows that $c_p'$ is an $h$\nobreakdash-colouring of $T_p$ for which the number of vertices which receive colour $i$ is $b_i^q+b_i^r=a^p_i$ for every $i\in[m]$ and which has less monochromatic edges than $c_p$, a contradiction.
        It follows that $f^p(a^p)$ is the number of monochromatic edges of $c_p$ and thus that the claim holds for $p$. 

        We have shown that for every node $p\in V(T)$ and every $h$-tuple $a^p$, we can compute in polynomial time the smallest number of monochromatic edges of all $h$\nobreakdash-colourings of $T_p$ whose $i$-th colour class has size $a_i^p$ for each $i\in[h]$.
        In particular, the value of $f^s(a^s)$, where $s$ is the root of $T$, is the minimum number of monochromatic edges of all $h$-colourings of $T$ whose $i$-th colour class has size $a_i^s$ for each $i\in[h]$.
        Since there are $O(n^h)$ possibilities how to fix the sizes of $h$ colour classes which partition $n$ vertices, we can determine the minimum number of monochromatic edges of $h$-colourings for each of them and then take the minimum of all of them.
        This can be done in polynomial time and shows thus the theorem.

        
\end{proof}

For the next theorem we will need the following definitions.
Let $\ell$ be a positive integer, $a= (a_1,\ldots,a_\ell),a' = (a_1',\ldots,a_\ell')$ two $\ell$-tuples of real numbers and $\lambda\leq\ell$ a positive integer.
We want to associate $\lambda$ entries of one tuple to $\lambda$ entries of the other tuple.

We say that a \emph{$\lambda$-matching $\mu$ of two $\ell$-tuples} is a $\lambda$-tuple $((i_1,i'_1), \dots, (i_\lambda,i'_\lambda)) \in ([\ell]\times [\ell])^\lambda$ such that for all $j,j' \in [\lambda]$, $j \neq j'$, we have that $i_j \neq i_{j'}$ and $i'_j \neq i'_{j'}$.
Intuitively, we can imagine that a $\lambda$-matching encodes the edges of a matching in a complete bipartite graph whose vertices correspond to the entries of two $\ell$-tuples, see Figure~\ref{fig-matching} for an example.
For any $\mu_j = (i_j, i'_j)$, $j \in [\lambda]$, we say that the index $i_j$ is matched with the index $i'_j$.
For the $\ell$-tuples $a$ and $a'$ we say that the entry $a_{i_j}$ is matched with the entry~$a'_{i'_j}$.
The \emph{value of a $\lambda$-matching $\mu$ applied to the $\ell$-tuples $a$ and $a'$} is denoted by $\val(\mu,a,a')$ and is defined as $\sum_{j=1}^\lambda a_{i_j}a'_{i'_j}$.

We now want to create a new tuple from $a$ and $a'$ and a $\lambda$-matching $\mu$.
A \emph{$\mu$-merge} of $a$ and $a'$ is a $(2\ell-\lambda)$-tuple whose $j$-th entry is $a_{i_j}+a'_{i'_j}$ for all $j\in[\lambda]$, the next $\ell-\lambda$ entries contain the unmatched entries of $a$, and the last $\ell-\lambda$ entries contain the unmatched entries of $a'$, see Figure~\ref{fig-matching} for an example.
We obtain a \emph{sorted $\mu$-merge} of $a$ and $a'$ if we sort the entries of any $\mu$-merge of $a$ and $a'$ in ascending order.
Observe that there might be several $\mu$\nobreakdash-merges of $a$ and $a'$, but there is only one sorted $\mu$-merge.
Given a $\lambda$-matching~$\mu$ and two $\ell$-tuples $a$ and $a'$, we denote the sorted $\mu$-merge of $a$ and $a'$ by $\merge(\mu,a,a')$.
We denote the $i$-th entry of $\merge(\mu, a, a')$ by $\merge(\mu, a, a', i)$, for every $i \in [2\ell-\lambda]$.

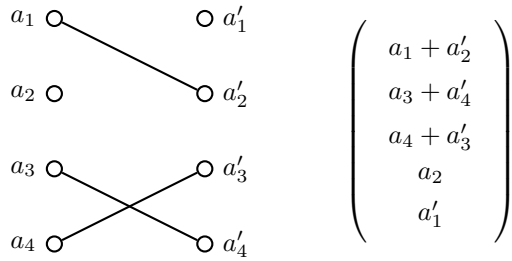
\begin{figure}
\centering
\begin{tikzpicture}

\node[vertex, label=left:$a_1$](a1) at (0,4){};
\node[vertex, label=left:$a_2$](a2) at (0,3){};
\node[vertex, label=left:$a_3$](a3) at (0,2){};
\node[vertex, label=left:$a_4$](a4) at (0,1){};

\node[vertex, label=right:$a'_1$](ap1) at (2,4){};
\node[vertex, label=right:$a'_2$](ap2) at (2,3){};
\node[vertex, label=right:$a'_3$](ap3) at (2,2){};
\node[vertex, label=right:$a'_4$](ap4) at (2,1){};

\draw[hedge](a1)--(ap2);
\draw[hedge](a3)--(ap4);
\draw[hedge](a4)--(ap3);

\begin{scope}[shift = {(5,2.5)}]


  \matrix [matrix of math nodes,left delimiter=(,right delimiter=)]
  {
    a_1 + a_2' \\
    a_3 + a_4' \\
    a_4 + a_3' \\
    a_2 \\
    a_1' \\
  };

\end{scope}

\end{tikzpicture}

\caption{\label{fig-matching}Left: The $3$-matching $\mu = ((1,2),(3,4),(4,3))$ visualized on the $4$-tuples $a = (a_1,\dots, a_4)$ and $a' = (a_1',\dots, a_4')$. Right: A corresponding $\mu$-merge.}
\end{figure}

\begin{property}\label{property1}
    Let $G$ be a cograph with corresponding cotree $T$.
    A colouring $c$ of $T$ is said to have \emph{Property~\ref{property1}} if for every 0-node $p$ of $T$ with children $q$ and $r$, the $i$-th largest colour class of $T_q$ has the same colour as the $i$-th largest colour class of $T_r$ for all $i \in [\chi(T_p)]$.
\end{property}

\begin{lemma} \label{lem-structcolouring}
Let $G$ be a cograph with associated cotree $T$. Let $d$ be a non-negative integer.
There is a $(\chi(G)-d)$-colouring $c$ of $G$ with the minimum possible number of monochromatic edges which satisfies Property~\ref{property1}.
\end{lemma}
\begin{proof}
    Let $p\in V(T)$ be a 0-node with children $q$ and $r$. 
    Suppose that there is an $i\in[\chi(T_p)]$ such that the $i$-th largest colour class of $T_q$ does not have the same colour as the $i$-th largest colour class of $T_r$ and assume that $i$ is the smallest such number.
    Note that the $i$-th colour classes of both $T_q$ and $T_r$ cannot be empty since an empty colour class has the same colour as every other colour class.
    By construction of the cotree, all vertices in $T_p$ have the same neighbourhood in $T_{\overline{p}}$.
    Denote by $n_q$ and $n_r$ the number of vertices in $T_{\overline{p}}$ which have the same colour as the vertices in the $i$-th largest colour class of $T_q$ and $T_r$, respectively, and which are adjacent to the vertices in $T_p$.
    Assume, without loss of generality, that $n_q\geq n_r$.
    Denote by $a_q$ and $a_r$ the number of vertices of the $i$-th largest colour class of $T_q$ and $T_r$, respectively.
    Denote by $a'_q$ the number of vertices in $T_q$ which have the colour of the $i$-th largest colour class of $T_r$ and denote by $a'_r$ the number of vertices which have the colour of the $i$-th largest colour class of $T_q$.
    Observe that $a'_q\leq a_q$ and $a'_r\leq a_r$ since we assumed that $i$ is the smallest index for which the $i$-th largest colour classes of $T_q$ and $T_r$ do not have the same colour.
    There are $n_q(a_q+a'_r)+n_r(a_r+a'_q)$ monochromatic edges between $T_p$ and $T_{\overline{p}}$ whose endpoints have the colours of the $i$-th colour classes of $T_q$ or $T_r$.
    Recolour the vertices in the $i$-th largest colour class of $T_q$ with the colour of the $i$-th largest colour class of $T_r$ and colour the vertices in $T_q$ which have the colour of the $i$-th largest colour class in $T_r$ with the colour of the $i$-th largest colour class of $T_q$.
    Observe that we only exchanged the colours of two colour classes in $T_q$, so we can only have changed the number of monochromatic edges between $T_p$ and $T_{\overline{p}}$ which have the colours of the $i$-th largest colour classes of $T_p$ and $T_q$.
    After recolouring, this number is now $n_r(a_q+a_r)+n_q(a'_q+a'_r)$.
    The difference of the number of monochromatic edges is $n_q(a_q+a'_r)+n_r(a_r+a'_q) - n_r(a_q+a_r)+n_q(a'_q+a'_r) = (n_r-n_q)(a'_q-a_q)$, which is the product of two non-positive numbers and thus non-negative.
    This shows that the recolouring did not increase the number of monochromatic edges.
    Repeating this step for each $i\in[\chi(T_p)]$ as above yields a colouring which satisfies the condition above.

    
    Observe that if the condition above holds for any descendant $q$ of $p$ in the original colouring, then it still holds after the recolouring, since any two vertices in $T_q$ having the same colour in the original colouring will still have the same colour after the recolouring step.
    We can thus apply the recolouring procedure starting at the leaves and working our way up to the root, ending up with a colouring of the graph for which the desired condition holds for every node of the cotree.

\end{proof}

\begin{lemma}\label{1node}
    Let $G$ be a cograph with associated cotree $T$ whose root $p$ is a 1-node with children $q$ and $r$.
    Let $a^p=(a^p_1,\ldots,a^p_\ell)$ be an $\ell$-tuple of non-negative integers and let $\lambda,\Delta$ be non-negative integers.
    There is a $(\chi(T_p)-\Delta)$-colouring of $T_p$ with Property~\ref{property1} whose $i$-th smallest colour class has at most $a^p_i$ vertices for every $i\in[\ell]$, which has exactly $\lambda$ colours which appear in both $T_q$ and $T_r$, which has the minimum number of monochromatic edges amongst all such colourings and there is no $j>\ell+\lambda$ such that the colour of the $j$-th smallest colour class of $T_q$ (or $T_r$) appears in $T_r$ (or $T_q$, respectively).
\end{lemma}
\begin{proof}
    Let $c_p$ be a $(\chi(T_p)-\Delta)$-colouring of $T_p$ with Property~\ref{property1} such that its $i$-th smallest colour class has at most $a^p_i$ vertices for every $i\in[\ell]$, the number of colours which appear in both $T_q$ and $T_r$ is $\lambda$ and its number of monochromatic edges amongst all such colourings is minimal.
    If there is no $j>\ell+\lambda$ such that the $j$-th smallest colour class of $T_q$ shares its colour with some vertices in $T_r$, or vice versa, we are done.
    If not, suppose, without loss of generality, that there are $j>\ell+\lambda$ and $j'$ such that the $j$-th smallest colour class of $T_q$ and the $j'$-th smallest colour class of $T_r$ have the same colour.
    Besides the $j$-th smallest colour class, there are $\lambda-1$ other colour classes in $T_q$ whose colour also appears in $T_r$.
    Since $j\geq\ell+\lambda+1$, there is an $i\in[\ell+\lambda]$ such that the $i$-th smallest colour class in $T_q$ has a colour which does not appear in $T_r$ and which is not one of the $\ell$ smallest colour classes of $T_p$.
    Recolour the vertices of the $i$-th and the $j$-th smallest colour classes of $c_q$ by exchanging their colours.
    This does not change the number of monochromatic edges with both ends in $T_q$.
    The number of monochromatic edges between $T_q$ and $T_r$ cannot increase since the $i$-th smallest colour class does not contain more vertices than the $j$-th smallest one.
    It follows that the new colouring did not increase the number of monochromatic edges or the number of colours used and it did not change the sizes of the $\ell$ smallest colour classes.
    It is also clear that the new colouring still has Property~\ref{property1}.
    Repeating this process for every $j$ as above yields a colouring as desired.
\end{proof}


\begin{theorem}\label{bigtheorem}
For a fixed integer $d$, $(\chi(G)-d)$-\textsc{Monochromatic Edges} is solvable in polynomial time on cographs.
\end{theorem}
\begin{proof}

        Let $G$ be a cograph with associated cotree $T$ and let $d\leq\chi(G)$ be a fixed integer.
        For every $\ell\in [0..d]$ and every node $p\in V(T)$ we define a function $f_\ell^p$ which takes as input an $\ell$-tuple of non-negative integers $a^p = (a^p_1,\ldots,a^p_\ell)$ with $a^p_i\leq a^p_j$ for every $i\leq j$ and a non-negative integer $\Delta$ with $\Delta\leq d-\ell$.

        
        If $p$ is a leaf, then $f_\ell^p$ will output 0 for every $\ell$ and every input.
        If $p$ is not a leaf, let $q$ and $r$ be the children of $p$.
        Let $\ell$ and $\Delta$ be non-negative integers with $\ell + \Delta \leq d$ and let $a^p = (a^p_1,\ldots,a^p_\ell)$ be an $\ell$-tuple as above.
        
        If $p$ is a 0-node, then assume, without loss of generality, that $\chi(T_q)\geq\chi(T_r)$ and let $\delta=\chi(T_q)-\chi(T_r)$.
        We distinguish the following three cases:

        \textbf{Case 1:} If $\Delta\geq\delta$ we set
          $$f_\ell^p(a^p,\Delta)=\min_{\substack{
        a^q,a^r\in[0..n]^{\ell}\\
        a_i^q+a_i^r\leq a^p_i, i\in[\ell]}}f_\ell^q(a^q,\Delta)+f_{\ell}^r(a^r,\Delta-\delta).$$

           \textbf{Case 2:} If $\Delta<\delta$ and $\Delta+\ell\geq\delta$ we set
          $$f_\ell^p(a^p,\Delta)=\min_{\substack{
        a^q, a^r\in[0..n]^{\ell-\delta+\Delta}\\
        a_{\delta-\Delta+i}^q+a_i^r\leq a^p_{\delta-\Delta+i}, i\in[\ell-\delta+\Delta]}}f_\ell^q(a^p_1,\ldots,a^p_{\delta-\Delta},a^q,\Delta)+f_{\ell-\delta+\Delta}^r(a^r,0).$$

        \textbf{Case 3:} If $\Delta<\delta$ and $\Delta+\ell<\delta$ we set
          $$f_\ell^p(a^p,\Delta)=f_\ell^q(a^p,\Delta).$$

        If $p$ is a 1-node, then 
        $$f_\ell^p(a^p,\Delta)=\min_{\substack{
        \Delta_q,\Delta_r,\lambda\geq 0\\
        \Delta_q+\Delta_r+\lambda=\Delta\\
        a^q, a^r\in[0..n]^{\ell+\lambda}\\
        \mu\textrm{ $\lambda$-matching of }(\ell+\lambda)\textrm{-tuples}\\
        \merge(\mu,a^q,a^r,i)\leq a^p_i,i\in[\ell]}}
        f_{\ell+\lambda}^q(a^q,\Delta_q)+f_{\ell+\lambda}^r(a^r,\Delta_r)+ \val(\mu,a^q,a^r).$$

        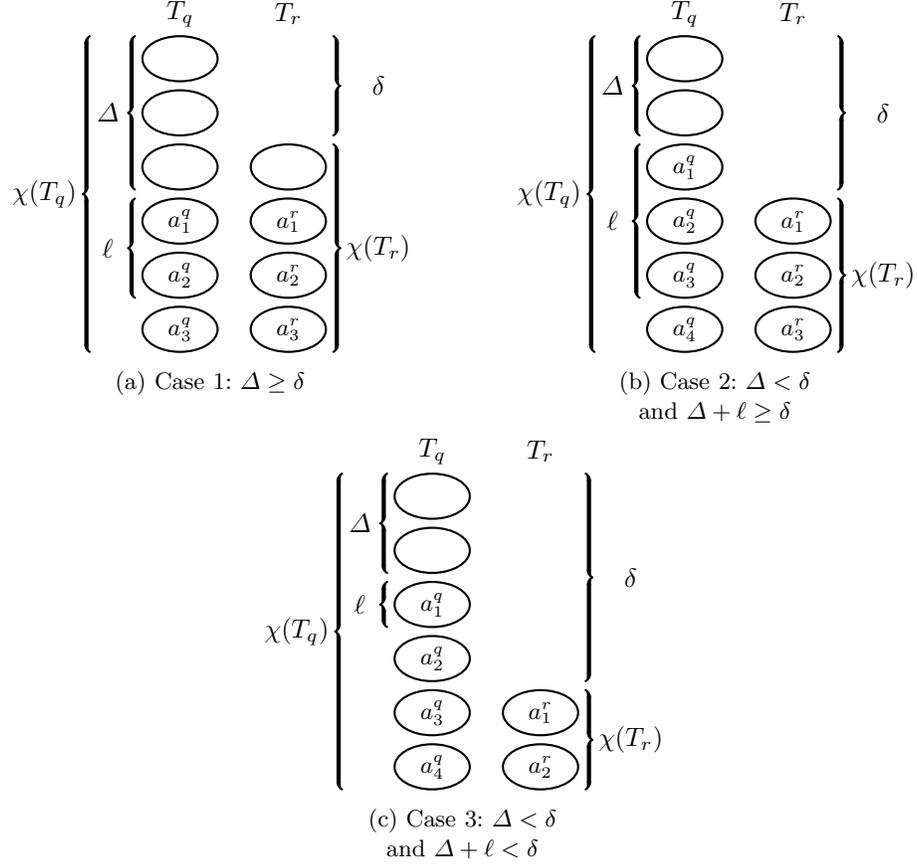
\begin{figure}[!ht]
        \centering
        \begin{subfigure}[t]{0.45\textwidth}
        \centering
        \begin{tikzpicture}
            \tikzstyle{class}=[thick,ellipse, minimum width=10mm, minimum height = 6mm, fill=white, draw=black, inner sep = 0pt]
            \tikzstyle{br} = [decorate, ultra thick, decoration = {calligraphic brace}]

        \begin{scope}[scale = 1.2]
            \footnotesize

            \node[class](a) at (0,0){$a^q_3$};
            \node[class](a) at (0,0.6){$a^q_2$};
            \node[class](a) at (0,1.2){$a^q_1$};
            \node[class](a) at (0,1.8){};
            \node[class](a) at (0,2.4){};
            \node[class](a) at (0,3){};

            \node[class](a) at (1.2,0){$a^r_3$};
            \node[class](a) at (1.2,0.6){$a^r_2$};
            \node[class](a) at (1.2,1.2){$a^r_1$};
            \node[class](a) at (1.2,1.8){};

            \normalsize

            \draw [br] (1.7,3.25) --  (1.7,2.15);
            \node[](delta) at (2.2, 2.7){$\delta$};

            \draw [br] (-0.5,1.55) --  (-0.5,3.25);
            \node[](Delta) at (-0.8, 2.4){$\Delta$};

            \draw [br] (-0.5,0.35) --  (-0.5,1.45);
            \node[](l) at (-0.8, 0.9){$\ell$};

            \draw [br] (-1,-0.25) --  (-1,3.25);
            \node[](xtq) at (-1.5, 1.5){$\chi(T_q)$};

            \draw [br] (1.7,2.05) --  (1.7,-0.25);
            \node[](xtr) at (2.2, 0.9){$\chi(T_r)$};

            \node[](tq) at (0, 3.5){$T_q$};
            \node[](tr) at (1.2, 3.5){$T_r$};
            \end{scope}
            
        \end{tikzpicture}
        \centering \caption{Case 1: $\Delta \geq \delta$}
        \end{subfigure}
        \hfill
        \begin{subfigure}[t]{0.45\textwidth}
        \begin{tikzpicture}
            \tikzstyle{class}=[thick,ellipse, minimum width=10mm, minimum height = 6mm, fill=white, draw=black, inner sep = 0pt]
            \tikzstyle{br} = [decorate, ultra thick, decoration = {calligraphic brace}]

        \begin{scope}[scale = 1.2]
            \footnotesize
            \node[class](a) at (0,0){$a^q_4$};
            \node[class](a) at (0,0.6){$a^q_3$};
            \node[class](a) at (0,1.2){$a^q_2$};
            \node[class](a) at (0,1.8){$a^q_1$};
            \node[class](a) at (0,2.4){};
            \node[class](a) at (0,3){};

            \node[class](a) at (1.2,0){$a^r_3$};
            \node[class](a) at (1.2,0.6){$a^r_2$};
            \node[class](a) at (1.2,1.2){$a^r_1$};
            \normalsize

            \draw [br] (1.7,3.25) --  (1.7,1.55);
            \node[](delta) at (2.2, 2.4){$\delta$};

            \draw [br] (-0.5,2.15) --  (-0.5,3.25);
            \node[](Delta) at (-0.8, 2.7){$\Delta$};

            \draw [br] (-0.5,0.35) --  (-0.5,2.05);
            \node[](l) at (-0.8, 1.2){$\ell$};

            \draw [br] (-1,-0.25) --  (-1,3.25);
            \node[](xtq) at (-1.5, 1.5){$\chi(T_q)$};

            \draw [br] (1.7,1.45) --  (1.7,-0.25);
            \node[](xtr) at (2.2, 0.6){$\chi(T_r)$};

            \node[](tq) at (0, 3.5){$T_q$};
            \node[](tr) at (1.2, 3.5){$T_r$};
        \end{scope}
            
        \end{tikzpicture}
        \caption{Case 2: $\Delta < \delta$ \\ and $\Delta + \ell \geq \delta$}
        \end{subfigure}
        \hfill
        \begin{subfigure}[t]{0.45\textwidth}
        
        \begin{tikzpicture}
            \tikzstyle{class}=[thick,ellipse, minimum width=10mm, minimum height = 6mm, fill=white, draw=black, inner sep = 0pt]
            \tikzstyle{br} = [decorate, ultra thick, decoration = {calligraphic brace}]

        \begin{scope}[scale = 1.2]
                \footnotesize
            \node[class](a) at (0,0){$a^q_4$};
            \node[class](a) at (0,0.6){$a^q_3$};
            \node[class](a) at (0,1.2){$a^q_2$};
            \node[class](a) at (0,1.8){$a^q_1$};
            \node[class](a) at (0,2.4){};
            \node[class](a) at (0,3){};

            \node[class](a) at (1.2,0){$a^r_2$};
            \node[class](a) at (1.2,0.6){$a^r_1$};

            \normalsize
            \draw [br] (1.7,3.25) --  (1.7,0.95);
            \node[](delta) at (2.2, 2.1){$\delta$};

            \draw [br] (-0.5,2.15) --  (-0.5,3.25);
            \node[](Delta) at (-0.8, 2.7){$\Delta$};

            \draw [br] (-0.5,1.55) --  (-0.5,2.05);
            \node[](l) at (-0.8, 1.8){$\ell$};

            \draw [br] (-1,-0.25) --  (-1,3.25);
            \node[](xtq) at (-1.5, 1.5){$\chi(T_q)$};

            \draw [br] (1.7,0.85) --  (1.7,-0.25);
            \node[](xtr) at (2.2, 0.3){$\chi(T_r)$};

            \node[](tq) at (0, 3.5){$T_q$};
            \node[](tr) at (1.2, 3.5){$T_r$};
        \end{scope}
            
        \end{tikzpicture}
        \caption{Case 3: $\Delta < \delta$ \\ and $\Delta + \ell < \delta$}
        \end{subfigure}
        
        \caption{\label{fig:cases}Illustration of the relation between $\delta, \Delta$ and $\ell$ in the three cases we consider in the proof.
        We represent a $(\chi(T_p)-\Delta)$-colouring with Property~\ref{property1} as a $\chi(T_p)$-colouring with Property~\ref{property1} whose smallest $\Delta$ colour classes are empty.
        Each ellipse represents a colour class in $T_q$ or $T_r$ and they are sorted ascendingly in size from top to bottom.
        By Property~\ref{property1}, two horizontally adjacent ellipses correspond to the same colour class in $T_p$.
        }
        \end{figure}

        This defines the values of $f_\ell^p$ for every $\ell\in[0..d]$ and $p\in V(T)$.
        Observe that for every $f_\ell^p$ there are at most $(n+1)^\ell (d-\ell+1)$ possible inputs.
        In order to compute the value of $f_\ell^p$ for one of them, we consider at most $(n+1)^{2(\ell+\Delta)}=O(n^{2d})$ possible pairs of tuples $a^q$ and $a^r$ if $p$ is a 0-node and $\Delta^3n^{2(\ell+\lambda)}\ell^{2\lambda}=O(n^{2d})$ choices for $\Delta_q, \Delta_r, \lambda, a^q, a^r$ and $\mu$ if $p$ is a 1-node.
        Thus, we can compute the function $f_\ell^p$ in polynomial time.

        \textbf{Claim.} For every $\ell\in[d], p\in V(T)$ and every valid input $(a^p, \Delta)$ the function $f_\ell^p(a^p, \Delta)$ returns the smallest number of monochromatic edges of all $(\chi(T_p)-\Delta)$-colourings of $T_p$ whose $i$-th smallest colour class has size at most $a^p_i$ for all $i \in [\ell]$ and which have Property~\ref{property1}.
        
        \textit{Proof of the Claim.}

        Let $p\in V(T)$ and let $\ell$ and $\Delta$ be non-negative integers with $\ell+\Delta\leq d$ and $a^p=(a^p_1,\ldots,a^p_\ell)\in [0..n]^\ell$.

        Observe first that the value of $f^p_\ell$ is correct when $p$ is a leaf since then there are no edges in $T_p$.
        If $p$ is not a leaf let $q$ and $r$ be the children of $p$ and assume that the claim holds for $q$ and $r$.

        Let $c_p$ be a $(\chi(T_p)-\Delta)$-colouring of $T_p$ which satisfies Property~\ref{property1} and whose $i$-th smallest colour class has at most $a^p_i$ vertices, for every $i \in [\ell]$, and which has the smallest number of monochromatic edges amongst all such colourings.
        We need to show that the number of monochromatic edges of $c_p$, which we denote by $m_p$, is identical to the value of $f_\ell^p(a^p,\Delta)$. 
        Let $c_q$ and $c_r$ be the colourings which we obtain by restricting $c_p$ to $T_q$ and $T_r$, respectively, and let $\Delta_q$ be an integer such that $\chi(T_q)-\Delta_q$ is the exact number of colours we use in $c_q$.
        We define $\Delta_r$ analogously.
        For every $i\in [\ell]$, denote by $a_i^q$ and $a_i^r$ the number of vertices of the $i$-th smallest colour class of $c_q$ and $c_r$, respectively.
        Observe that the $i$-th smallest colour classes of $c_q$ and $c_r$ do not necessarily have the same colour.
        Let $m_q$ and $m_r$ be the numbers of monochromatic edges of $c_q$ and $c_r$, respectively.
        
        Assume first that $p$ is a 0-node, set $\delta=\chi(T_q)-\chi(T_r)$ and assume without loss of generality that $\delta\geq 0$.
        Recall that $\chi(T_p) = \chi(T_q)$ and $m_q + m_r = m_p$, since $p$ is a 0-node.
        It follows from Property~\ref{property1} that for every $i>\chi(T_r)$ the $i$-th largest colour class of $c_p$ is entirely contained in $T_q$.
         To prove the claim we distinguish the following three cases, see Figure~\ref{fig:cases} for an illustration.
        
        \textbf{Case 1:} $\Delta\geq\delta$\\
        We first show that $f_\ell^p(a^p,\Delta)\leq m_p$.
        The colouring $c_p$ uses $\chi(T_q)-\Delta$ colours on $T_p$.
        We have that $\chi(T_q)-\Delta=\chi(T_r)-(\Delta-\delta)\leq\chi(T_r)$.
        It follows that $c_q$ is a $(\chi(T_q)-\Delta)$-colouring of $T_q$ and $c_r$ is a $(\chi(T_r)-(\Delta-\delta))$-colouring of $T_r$.
        By Property~\ref{property1} we know that the $i$-th largest colour class of $c_p$ on $T_p$ consists of the $i$-th largest colour class of $c_q$ and the $i$-th largest colour class of $c_r$ for every $i \in [\chi(T_q)-\Delta]$.
        This implies in particular that $a_i^p\geq a_i^q+a_i^r$ for every $i \in [\ell]$.
        Since $c_q$ is a $(\chi(T_q)-\Delta)$-colouring of $T_q$ whose $i$-th smallest colour class contains $a_i^q$ vertices and since the claim holds for $q$, it follows that $f_\ell^q(a^q,\Delta)\leq m_q$ and analogously $f_\ell^r(a^r,\Delta-\delta)\leq m_r$.
        Hence $f_\ell^p(a^p,\Delta) \leq f_\ell^q(a^q,\Delta) + f_\ell^r(a^r,\Delta-\delta)\leq m_q + m_r = m_p$.

        It remains to show that $f_\ell^p(a^p,\Delta)= m_p$.
        We suppose for a contradiction that there are $a^q, a^r\in[0..n]^\ell$ such that $a_i^q+a_i^r\leq a_i^p$ for all $i\in[\ell]$ and $f_\ell^p(a^p,\Delta) =f_\ell^q(a^q,\Delta)+f_\ell^r(a^r,\Delta-\delta)< m_p$.
        Since the claim holds for $q$ and $r$, it follows that there is a $(\chi(T_q)-\Delta)$-colouring $c'_q$ of $T_q$ and a $(\chi(T_r)-(\Delta-\delta))$-colouring $c'_r$ of $T_r$, both with Property~\ref{property1}, such that the $i$-th smallest colour class of $c'_q$ ($c'_r$, respectively) contains at most $a_i^q$ vertices ($a_i^r$ vertices, respectively) for each $i\in[\ell]$ and that the sum of the number of monochromatic edges in $c_q'$ and $c_r'$ is strictly less than $m_p$.
        Construct a $(\chi(T_p)-\Delta)$-colouring $c_p'$ of $T_p$ with Property~\ref{property1} such that the $i$-th largest colour class of $c_p'$ contains exactly the vertices of the $i$-th largest colour class of $c'_q$ and the vertices of the $i$-th largest colour class of $c_r'$ for every $i\in[\chi(T_q)-\Delta]$.
        It follows that the size of the $i$-th smallest colour class of $c_p'$ is at most $a_i^p$ for every $i\in[\ell]$.
        Since $p$ is a 0-node, the number of monochromatic edges of $c_p'$ is exactly the sum of the numbers of monochromatic edges of $c'_q$ and $c'_r$ and thus less than $m_p$.
        By construction, $c_p'$ has Property~\ref{property1} and so $c_p'$ is a $(\chi(T_p)-\Delta)$-colouring with Property~\ref{property1} of $T_p$ whose $i$-th smallest colour class has size at most $a_i^p$ for every $i\in[\ell]$ and which has less monochromatic edges than $c_p$, a contradiction to the choice of $c_p$.

        \textbf{Case 2:} $\Delta<\delta$ and $\Delta+\ell\geq\delta$\\
        Again we first show that $f_\ell^p(a^p,\Delta)\leq m_p$.
        Since $c_p$ uses $\chi(T_q)-\Delta$ colours, it follows from Property~\ref{property1} that the smallest $\chi(T_q)-\Delta-\chi(T_r)=\delta-\Delta$ colour classes of $c_p$ are entirely contained in $T_q$ and thus $a_i^q\leq a_i^p$ for each $i\in[\delta-\Delta]$.
        Further, $c_r$ uses $\chi(T_r)$ colours.
        For every $i\in\set{\delta-\Delta+1,\ldots,\ell}$, the $i$-th smallest colour class of $c$ consists of the $i$-th smallest colour class of $c_q$ and the $(i-\delta+\Delta)$-th smallest colour class of $c_r$ and thus $a_{i}^p\geq a_i^q+a^r_{i-\delta+\Delta}$.
        Hence, $f_\ell^q(a^p_1,\ldots,a^p_{\delta-\Delta},a_{\delta-\Delta+1}^q,\ldots,a^q_\ell,\Delta)\leq f_\ell^q(a^q,\Delta) \leq m_q$ and $f^r_{\ell-\delta+\Delta}(a^r,0)\leq m_r$.
        It follows that $f_\ell^p(a^p,\Delta) \leq f_\ell^q(a_1^p,\ldots,a^p_{\delta-\Delta},a_{\delta-\Delta+1}^q,\ldots,a^q_\ell,\Delta) + f^r_{\ell-\delta+\Delta}(a^r,0) \leq m_q + m_r = m_p$.

        It remains to show that $f_\ell^p(a^p,\Delta)= m_p$.
        Again, we suppose for a contradiction that there are $a^q, a^r\in[0..n]^{\ell-\delta+\Delta}$ such that $a_{i+\delta-\Delta}^q+a_i^r\leq a^p_{i+\delta-\Delta}$ for all $i\in[\ell-\delta+\Delta]$ and $f_\ell^q(a^p_1,\ldots,a^p_{\delta-\Delta},a^q,\Delta)+f_{\ell-\delta+\Delta}^r(a^r,0)<m_q+m_r$.
         Since we assumed that the claim holds for $q$ and $r$ it follows that there is a $(\chi(T_q)-\Delta)$-colouring $c'_q$ of $T_q$ and a $\chi(T_r)$-colouring $c'_r$ of $T_r$, both with Property~\ref{property1}, such that the $i$-th smallest colour class of $c'_q$ contains at most $a^p_i$ vertices for every $i\in[\delta-\Delta]$, the $(i+\delta-\Delta)$\nobreakdash-th smallest colour class of $c'_q$ contains at most $a^q_{i+\delta-\Delta}$ vertices for each $i\in[\ell-\delta+\Delta]$ and the $i$-th smallest colour class of $c_r'$ contains at most $a_i^r$ vertices for each $i\in[\ell-\delta+\Delta]$ and that the sum of monochromatic edges in $c'_q$ and $c'_r$ adds to less than $m_q+m_r$.
         

        Construct a $(\chi(T_q)-\Delta)$-colouring $c_p'$ of $T_p$ with Property~\ref{property1} such that the $i$-th smallest colour class of $c_p'$ contains exactly the vertices of the $i$-th smallest colour class of $c'_q$ for every $i\in[\delta-\Delta]$ and the $i$-th smallest colour class of $c_p'$ contains exactly the vertices of the $i$-th smallest colour class of $c'_q$ and the vertices of the $(i-\delta+\Delta)$-th smallest colour class of $c'_r$ for every $i\in\set{\delta-\Delta+1,\ldots,\chi(T_q)-\Delta}$.
        It follows that the size of the $i$-th smallest colour class of $c_p'$ has size at most $a^p_i$ for every $i\in[\ell]$.
        It is also clear that the number of monochromatic edges of $c_p'$ is exactly the sum of the numbers of monochromatic edges of $c'_q$ and $c'_r$ and thus less than $m_q+m_r$.
        So $c_p'$ is a $(\chi(T_q)-\Delta)$-colouring of $T_p$ with Property~\ref{property1} whose $i$-th colour class has size at most $a^p_i$ for every $i\in[\ell]$ and which has less monochromatic edges than $c_p$, a contradiction to the choice of $c_p$.
        
        \textbf{Case 3:} $\Delta<\delta$ and $\Delta+\ell<\delta$\\
        In this case, the colours used in the $\ell$ smallest colour classes of $c$ do not appear in $T_r$. 
        Clearly, $f_\ell^r(a^r,\Delta) = 0$ and thus $f_\ell^p(a^p, \Delta) \leq f_\ell^q(a^q,\Delta) + f_\ell^r(a^r,\Delta)\leq m_q$.

        Suppose for a contradiction that $f_\ell^q(a^q,\Delta)<m_q$.
        Then there is a $(\chi(T_q)-\Delta)$-colouring $c_q'$ of $T_q$ with Property~\ref{property1} such that the number of vertices in the $i$-th smallest colour class of $c_q'$ contains at most $a^p_i$ vertices for each $i\in[\ell]$ and the number of monochromatic edges of $c'_q$ is less than $m_q$.
        Let $c'_r$ be a proper colouring of $T_r$ with Property~\ref{property1} using $\chi(T_r)$ colours.
        Construct a $(\chi(T_p)-\Delta)$-colouring $c_p'$ of $T_p$ with Property~\ref{property1} such that the $i$-th smallest colour class of $c_p'$ contains exactly the vertices of the $i$-th smallest colour class of $c'_q$ for each $i\in[\delta-\Delta]$ and that the $i$-th smallest colour class of $c_p'$ contains exactly the vertices of the $i$-th smallest colour class of $c'_q$ and the $(i-\delta+\Delta)$-th smallest colour class of $c'_r$ for $i \in \set{\delta-\Delta +1 ,\ldots, \chi(T_r)-\Delta}$.
        Since $\ell<\delta-\Delta$, it follows that the $\ell$ smallest colour classes of $c_p'$ are all contained in $T_q$ and thus $c_p'$ is a $(\chi(T_p)-\Delta)$-colouring of $T_p$ whose $i$-th smallest colour class contains at most $a^p_i$ vertices for each $i\in[\ell]$.
        Since the number of monochromatic edges of $c_p'$ is equal to the number of monochromatic edges of $c'_q$, it follows that $c_p'$ has less monochromatic edges than $c_p$, a contradiction.
        This completes the third case and thus the claim holds for $p$ if $p$ is a 0-node.
        
       Assume that $p$ is a 1-node.
       Let $\lambda$ be the number of colours that are used by both $c_q$ and $c_r$.
       Following \Cref{1node} we can assume that only the $\ell+\lambda$ smallest colour classes have colours which appear in both $T_q$ and $T_r$.
       The number of colours used by~$c_p$ is $\chi(T_q)-\Delta_q+\chi(T_r)-\Delta_r-\lambda=\chi(T_p)-(\Delta_q+\Delta_r+\lambda)$, and so $\Delta_q+\Delta_r+\lambda=\Delta$.
       Let $\mu$ be a $\lambda$-matching on tuples of length $\ell+\lambda$ which matches $i$ with $j$ if and only if the $i$-th smallest colour class of $c_q$ and the $j$-th smallest colour class of $c_r$ have the same colour.
       It follows that $\merge(\mu,(a^q_1,\ldots,a^q_{\ell+\lambda}),(a^r_1,\ldots,a^r_{\ell+\lambda}),i)$ is the size of the $i$-th smallest colour class of $c_p$ for every $i\in[\ell]$.
       Observe that $c_q$ (and $c_r$, respectively) is a $(\chi(T_q)-\Delta_q)$-colouring of $T_q$ with Property~\ref{property1} (or a $(\chi(T_r)-\Delta_r)$-colouring of $T_r$ with Property~\ref{property1}, respectively) whose $i$-th colour class has exactly $a_i^q$ vertices ($a_i^r$ vertices, respectively) for every $i\in[\ell+\lambda]$ and thus, by the assumption that the claim holds for $q$ and $r$, $f_{\ell+\lambda}^q((a^q_1,\ldots,a^q_{\ell+\lambda}),\Delta_q)\leq m_q$ and $f_{\ell+\lambda}^r((a^r_1,\ldots,a^r_{\ell+\lambda}),\Delta_r)\leq m_r$.
       It is easy to see that the number of monochromatic edges of $c_p$ is $m_q+m_r+\val(\mu,a^q,a^r)$ and thus $f_\ell^p(a,\Delta)$ is smaller or equal than the number of monochromatic edges of $c_p$.

       Suppose for a contradiction that $f_\ell^p(a,\Delta)$ is strictly smaller than the number of monochromatic edges of $c_p$.
       This implies that there are 
       \begin{itemize}
           \item $\Delta'_q,\Delta'_r,\lambda'\geq 0$ with $\Delta'_q+\Delta'_r+\lambda'=\Delta$,
           \item tuples $b^q,b^r\in[0..n]^{\ell+\lambda'}$,
           \item a $\lambda'$-matching $\mu'$ of $(\ell+\lambda')$-tuples such that $\merge(\mu',b^q,b^r)_i\leq a_i^p$ for each $i\in[\ell]$ and
           \item $f_{\ell+\lambda}^q(b^q,\Delta'_q)+f_{\ell+\lambda}^r(b^r,\Delta'_r)+val(\mu',b^q,b^r)$ is strictly less than the number of monochromatic edges of $c_p$.
       \end{itemize}
       By the assumption of correctness of the claim for $q$ and $r$, there are
       \begin{itemize}
           \item a $(\chi(T_q)-\Delta'_q)$-colouring $c'_q$ of $T_q$ with Property~\ref{property1},
           \item a $(\chi(T_r)-\Delta'_r)$-colouring $c'_r$ of $T_r$ with Property~\ref{property1}, such that
           \item the number of monochromatic edges of $c'_q$ and $c'_r$ are $m'_q$ and $m'_r$, respectively, 
           \item the $i$-th smallest colour class of $c'_q$ (or of $c'_r$, respectively) has at most $b_i^q$ vertices ($b_i^r$ vertices, respectively) for each $i\in[\ell+\lambda]$,
           \item $m'_q+m'_r+\val(\mu',b^q,b^r)$ is strictly less than the number of monochromatic edges of $c_p$.
       \end{itemize}
       Construct a $(\chi(T_p)-\Delta)$-colouring $c_p'$ of $T_p$ with Property~\ref{property1} such that two vertices in $T_q$ (in $T_r$, respectively) obtain the same colour if and only if they obtain the same colour by $c'_q$ (by $c'_r$, respectively) and two vertices $w_q\in T_q$ and $w_r\in T_r$, which are contained in the $i$-th smallest colour class of $c'_q$ and in the $i'$-th smallest colour class of $T_r$, respectively, obtain the same colour if and only if $i$ and $i'$ are both at most $\ell+\lambda'$ and $\mu'$ matches $i$ with $i'$.
       It follows immediately from the definition of this new colouring that $c_p'$ is a $(\chi(T_p)-\Delta)$-colouring of $T_p$ with Property~\ref{property1} such that its $i$-th smallest colour class contains at most $a^p_i$ vertices for every $i\in[\ell]$ and which has less monochromatic edges than $c_p$, a contradiction to the choice of $c_p$.

       This shows the claim.

       It follows from the correctness of the claim that $f_0^s((),d)$, where $s$ is the root of $T$ and the first argument is an empty tuple, is the smallest number of monochromatic edges that any $(\chi(G)-d)$-colouring of $G$ satisfying Property~\ref{property1} can have.
       It follows from \Cref{lem-structcolouring} that this coincides with the smallest number of monochromatic edges amongst \emph{all} $(\chi(G)-d)$-colourings of $G$.
       This shows that we can indeed solve $(\chi(G)-d)$-\textsc{Monochromatic Edges} in polynomial time.

\end{proof}

\subsection{Hardness proofs}

\begin{theorem}
    \textsc{Monochromatic Edges} is NP-hard on complete multipartite graphs.
\end{theorem}
\begin{proof}

    We reduce from \textsc{Minimum Sum of Squares}, which takes as input an integer~$\ell$, an $\ell$-tuple $a=(a_1,\ldots,a_{\ell})$ of integers, an integer $h$ and an integer $J$.
    It asks whether $[\ell]$ can be partitioned into $h$ sets $A_i$, $i\in [h]$, such that
    $$\sum_{i=1}^h\left(\sum_{j\in A_i}a_j\right)^2\leq J.$$
    It was shown in \cite{GareyJohnson} that this problem is NP-hard, when we consider $\sum_{j\in [\ell]}a_j$ as the size of an instance.
    Given an instance $(\ell,a,h,J)$ of \textsc{Minimum Sum of Squares}, we set $D=\frac{1}{2}\sum_{j=1}^\ell a_j^2$.
    We construct a complete multipartite graph $G$ as follows:
    for every $j\in [\ell]$, let $U_j$ be a set of $a_j$ vertices such that $U_j$ and $U_{j'}$ are disjoint for every $j,j'\in [\ell]$.
    Set $V(G)=\bigcup_{j\in [\ell]}U_j$ and let $E(G)$ be such that two vertices $v\in U_j$ and $w\in U_{j'}$ are adjacent if and only if $j\neq j'$.
    We claim that $(G,h,\frac{1}{2}J-D)$ is a \yes-instance for \textsc{Monochromatic Edges} if and only if $(\ell,a,h,J)$ is a \yes-instance for \textsc{Minimum Sum of Squares}.
    
    Assume first that $(G,h,\frac{1}{2}J-D)$ is a \yes-instance for \textsc{Monochromatic Edges}.
    We know from \Cref{Indepsets} that there is an $h$-colouring $c$ of $G$ minimizing the number of monochromatic edges such that $\vert c(U_j)\vert=1$ for every $j\in [\ell]$. 

    For every $i\in[h]$, let $A_i\subseteq [\ell]$ be such that $j\in [\ell]$ is contained in $A_i$ if and only if the vertices of $U_j$ are coloured  with colour $i$.
    Then the number of monochromatic edges whose both ends are coloured with colour $i$ equals $\frac{1}{2}(\sum_{j\in A_i}a_j)^2 - \sum_{j\in A_i}a_j^2 $ and the total number of monochromatic edges is
    \[\frac{1}{2} \sum_{i=1}^h\left( \left(\sum_{j\in A_i}a_j\right)^2 - \sum_{j\in A_i}a_j^2\right) =  \frac{1}{2}\sum_{i=1}^h \left(\sum_{j\in A_i}a_j\right)^2-D.\]
    Thus, \[\sum_{i=1}^h\left(\sum_{j\in A_i}a_j\right)^2 \leq 2\left(\frac{1}{2}J-D + D\right) = J.\]
    It follows that $(\ell,a,h,J)$ is a \yes-instance for \textsc{Minimum Sum of Squares}.
    
    Assume now that $(\ell,a,h,J)$ is a \yes-instance for \textsc{Minimum Sum of Squares}.
    For every $j \in A_i$, $i \in [h]$, we colour the corresponding vertex set $U_j$ in $G$ with colour~$i$. 
    As above, the number of monochromatic edges between vertices with colour $i$ is $\frac{1}{2}\sum_{i=1}^h \left(\sum_{j\in A_i}a_j\right)^2-D\leq \frac{1}{2}J-D$. 

    Thus, $(G,h,\frac{1}{2}J-D)$ is a \yes-instance for \textsc{Monochromatic Edges}.
    

\end{proof}

\section*{Acknowledgement}
The authors thank Carolina Luc\'{i}a Gonzalez for helpful discussions about \Cref{bigtheorem}.

\section*{Declarations of competing interests}

The authors have no relevant financial or non-financial interests to disclose.

%
%

\bibliographystyle{splncs04}
\bibliography{sn-bibliography}
\end{document}